\documentclass[11pt]{article}
\usepackage[margin=1in]{geometry}
\usepackage{amsmath,amssymb,amsthm}
\usepackage{enumitem}
\usepackage{hyperref}
\usepackage[dvipsnames]{xcolor}
\usepackage{tikz}
\usetikzlibrary{decorations.pathreplacing,arrows.meta,calc,patterns,positioning}
\usepackage{booktabs}
\usepackage{graphicx}
\usepackage{subcaption}

\hypersetup{colorlinks=true, linkcolor=blue, citecolor=blue, urlcolor=blue}
\allowdisplaybreaks

\newtheorem{theorem}{Theorem}[section]
\newtheorem{proposition}[theorem]{Proposition}
\newtheorem{lemma}[theorem]{Lemma}
\newtheorem{corollary}[theorem]{Corollary}
\theoremstyle{definition}

\newcommand{\R}{\mathbb{R}}
\newcommand{\E}{\mathbb{E}}
\newcommand{\Prb}{\mathbb{P}}
\newcommand{\TT}{\mathbb{T}}
\newcommand{\Bbox}{\mathcal{B}}
\newcommand{\Dd}{\mathcal{D}}
\newcommand{\corereg}{\mathcal{D}_\mathrm{core}}
\newcommand{\RePart}{\operatorname{Re}}
\newcommand{\ImPart}{\operatorname{Im}}
\newcommand{\diag}{\operatorname{diag}}
\newcommand{\core}{G_{\mathrm{core}}}
\newcommand{\Ahat}{\hat A_{n,4t}}

\title{Counting partial Hadamard matrices in the cubic regime}
\author{Damek Davis\thanks{Department of Statistics and Data Science, The Wharton School, University of Pennsylvania; \texttt{damek@wharton.upenn.edu}. Research supported by NSF DMS award 2523384.}}
\date{}

\begin{document}
\maketitle

\begin{abstract}
We give a precise asymptotic formula for the number of \(n\times 4t\) partial Hadamard matrices in the regimes \(t/n^3\to\infty\) and \(t/n^3\to\Theta\) for sufficiently large fixed \(\Theta\). This strengthens earlier results of de~Launey and Levin, who obtained the asymptotic for \(t/n^{12}\to\infty\), and of Canfield, who extended this to \(t/n^4\to\infty\).
\end{abstract}

\section{Introduction}

The Hadamard conjecture asks whether, for every positive integer \(n\) divisible by~\(4\), there exists an \(n\times n\) matrix \(H\) with entries in \(\{\pm1\}\) satisfying \(HH^\top = nI\). Despite sustained effort since Sylvester's recursive construction \cite{Syl1867} and Paley's finite-field families \cite{Pal1933}, the conjecture remains open; we refer to Horadam~\cite{Hor2012} and Seberry--Yamada~\cite{SY2020} for background on Hadamard matrices, to Tressler~\cite{Tre2004} and Browne et al.~\cite{BEHOC2021} for surveys, and to Cati--Pasechnik~\cite{CP2024} for a recent computational database.

Rather than constructing full Hadamard matrices, one can study partial ones: an \(n\times t\) matrix with entries in \(\{\pm 1\}\) is a \emph{partial Hadamard matrix} if its rows are pairwise orthogonal. Partial Hadamard matrices of width nearly~\(2n\) are known to exist for large~\(n\) by combining constructions from combinatorial design theory~\cite{Crai1995,dLG2001} with analytic number theory~\cite{GS2008} (see \cite[Theorem~A]{DL10}). While these results settle existence for widths close to~\(2n\), a finer question is to determine how many partial Hadamard matrices there are as a function of~\(n\) and~\(t\).

De~Launey and Levin~\cite{DL10} introduced a Fourier-analytic framework that answers this question for large~\(t\). Their approach encodes the row-orthogonality conditions for an \(n\times t\) matrix through the \(d=\binom{n}{2}\) pairwise products \(y_iy_j\): writing \(Z(y)=(y_iy_j)_{1\le i<j\le n}\in\{\pm1\}^d\), a matrix \(Y=[y^{(1)}\cdots y^{(t)}]\) with \(\pm1\) entries is partial Hadamard if and only if \(\sum_{r=1}^t Z(y^{(r)})=0\); see Figure~\ref{fig:walk}. Counting partial Hadamard matrices therefore reduces to counting returns to the origin of the random walk \(S_t=\sum_{r=1}^t Z(\xi^{(r)})\) on~\(\mathbb{Z}^d\), where \(\xi^{(1)},\xi^{(2)},\dots\) are i.i.d.\ uniform on~\(\{\pm1\}^n\). If \(N_{n,t}\) denotes the number of \(n\times t\) partial Hadamard matrices,
\begin{equation}\label{eq:counting-identity}
  N_{n,t} \;=\; 2^{nt}\,\Prb(S_t=0),
\end{equation}
and the Fourier inversion formula \cite[P3, p.~57]{Spi1976} gives
\begin{equation}\label{eq:fourier-inversion}
  \Prb(S_t=0)
  \;=\;
  (2\pi)^{-d}\int_{\TT^d}\psi(\lambda)^t\,d\lambda,
  \qquad
  \psi(\lambda)=\E[e^{i\lambda\cdot Z(\xi)}].
\end{equation}
In particular, the Hadamard conjecture is equivalent to the positivity of the integral~\eqref{eq:fourier-inversion} at \(t=n\) for every $n$ divisible by $4$.

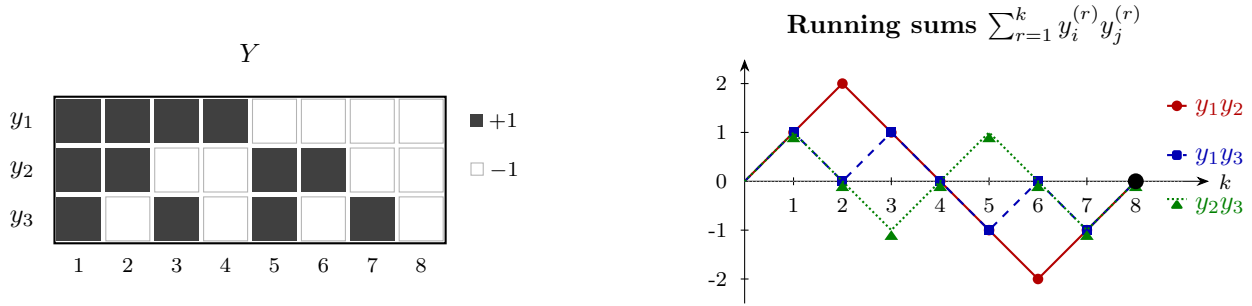
\begin{figure}[t]
\centering
\begin{tikzpicture}[x=0.65cm, y=0.65cm]
  \begin{scope}[yshift=-0.85cm]
  \node[above, font=\small\bfseries] at (4.5,4) {$Y$};
  \node[left, font=\small] at (0.3,3) {$y_1$};
  \node[left, font=\small] at (0.3,2) {$y_2$};
  \node[left, font=\small] at (0.3,1) {$y_3$};
  \foreach \c/\ra/\rb/\rc in {
    1/1/1/1, 2/1/1/0, 3/1/0/1, 4/1/0/0,
    5/0/1/1, 6/0/1/0, 7/0/0/1, 8/0/0/0} {
    \ifnum\ra=1
      \fill[black!75] (\c-0.45,2.55) rectangle ++(0.9,0.9);
    \else
      \fill[white] (\c-0.45,2.55) rectangle ++(0.9,0.9);
      \draw[gray!60] (\c-0.45,2.55) rectangle ++(0.9,0.9);
    \fi
    \ifnum\rb=1
      \fill[black!75] (\c-0.45,1.55) rectangle ++(0.9,0.9);
    \else
      \fill[white] (\c-0.45,1.55) rectangle ++(0.9,0.9);
      \draw[gray!60] (\c-0.45,1.55) rectangle ++(0.9,0.9);
    \fi
    \ifnum\rc=1
      \fill[black!75] (\c-0.45,0.55) rectangle ++(0.9,0.9);
    \else
      \fill[white] (\c-0.45,0.55) rectangle ++(0.9,0.9);
      \draw[gray!60] (\c-0.45,0.55) rectangle ++(0.9,0.9);
    \fi
  }
  \draw[thick] (0.5,0.5) rectangle (8.5,3.5);
  \foreach \c in {1,...,8} {
    \node[below, font=\scriptsize] at (\c,0.4) {\c};
  }
  \node[right, font=\scriptsize] at (8.8,3) {%
    \tikz{\fill[black!75] (0,0) rectangle (0.25,0.25);}\;$+1$};
  \node[right, font=\scriptsize] at (8.8,2) {%
    \tikz{\fill[white] (0,0) rectangle (0.25,0.25);
           \draw[gray!60] (0,0) rectangle (0.25,0.25);}\;$-1$};
  \end{scope}

  \begin{scope}[xshift=9.5cm, yshift=0.3cm]
    \node[above, font=\small\bfseries] at (4.5,2.6)
      {Running sums $\textstyle\sum_{r=1}^k y_i^{(r)}y_j^{(r)}$};
    \draw[-{Stealth}] (0,0) -- (9.5,0) node[right, font=\scriptsize] {$k$};
    \draw[-{Stealth}] (0,-2.5) -- (0,2.5) node[above, font=\scriptsize] {};
    \foreach \y in {-2,-1,1,2} {
      \draw (0.08,\y) -- (-0.08,\y);
      \node[left, font=\scriptsize] at (-0.15,\y) {\y};
    }
    \node[left, font=\scriptsize] at (-0.15,0) {$0$};
    \foreach \x in {1,...,8} {
      \draw (\x,0.08) -- (\x,-0.08);
      \node[below, font=\scriptsize] at (\x,-0.15) {\x};
    }
    \draw[gray!40, densely dotted] (0,0) -- (9,0);
    \draw[thick, red!70!black]
      (0,0) -- (1,1) -- (2,2) -- (3,1) -- (4,0)
      -- (5,-1) -- (6,-2) -- (7,-1) -- (8,0);
    \foreach \x/\y in {1/1,2/2,3/1,4/0,5/-1,6/-2,7/-1,8/0} {
      \fill[red!70!black] (\x,\y) circle (2pt);
    }
    \draw[thick, dashed, blue!70!black]
      (0,0) -- (1,1) -- (2,0) -- (3,1) -- (4,0)
      -- (5,-1) -- (6,0) -- (7,-1) -- (8,0);
    \foreach \x/\y in {1/1,2/0,3/1,4/0,5/-1,6/0,7/-1,8/0} {
      \fill[blue!70!black] (\x-0.1,\y-0.1) rectangle ++(0.2,0.2);
    }
    \draw[thick, densely dotted, green!50!black]
      (0,0) -- (1,1) -- (2,0) -- (3,-1) -- (4,0)
      -- (5,1) -- (6,0) -- (7,-1) -- (8,0);
    \foreach \x/\y in {1/1,2/0,3/-1,4/0,5/1,6/0,7/-1,8/0} {
      \fill[green!50!black] (\x,\y) -- ++(-0.14,-0.2) -- ++(0.28,0) -- cycle;
    }
    \node[right, font=\small, red!70!black,
      fill=white, inner sep=2pt, rounded corners=1pt] at (8.5,1.5)
      {\tikz{\draw[thick,red!70!black](0,0)--(0.4,0);
             \fill[red!70!black](0.2,0) circle(2pt);}\;$y_1 y_2$};
    \node[right, font=\small, blue!70!black,
      fill=white, inner sep=2pt, rounded corners=1pt] at (8.5,0.5)
      {\tikz{\draw[thick,dashed,blue!70!black](0,0)--(0.4,0);
             \fill[blue!70!black](0.1,-0.1) rectangle ++(0.2,0.2);}\;$y_1 y_3$};
    \node[right, font=\small, green!50!black,
      fill=white, inner sep=2pt, rounded corners=1pt] at (8.5,-0.5)
      {\tikz{\draw[thick,densely dotted,green!50!black](0,0)--(0.4,0);
             \fill[green!50!black](0.2,0)--++(-0.14,-0.2)--++(0.28,0)--cycle;}\;$y_2 y_3$};
    \fill[black] (8,0) circle (3pt);
  \end{scope}
\end{tikzpicture}
\caption{A \(3\times 8\) partial Hadamard matrix (left) and the running pairwise-product sums (right). The matrix is partial Hadamard exactly when all three sums return to zero at $k=8$.}\label{fig:walk}
\end{figure}

Using this framework, De~Launey and Levin~\cite[Theorem~2]{DL10} proved that for every fixed \(n\ge 2\), the number of \(n\times 4t\) partial Hadamard matrices satisfies
\begin{equation}\label{eq:fixed-n}
  N_{n,4t} \;=\; [1+o(1)]\,A_{n,4t}
  \qquad\text{as } t\to\infty,
\end{equation}
where \(A_{n,s}:=2^{ns+2d-n+1}(2\pi s)^{-d/2}\).
This gives both existence for large~\(t\) and a precise asymptotic count. Since the Hadamard conjecture is equivalent to \(N_{n,n}>0\) whenever $n$ is divisible by~$4$, one can also ask what happens when \(n\) grows with~\(t\), potentially at a slower rate: does \(N_{n,4t}\sim A_{n,4t}\) still hold, and if so, with what corrections? Two prior works address this in the regime \(t/n^\alpha\to\infty\):
\begin{enumerate}[label=(\roman*),nosep]
\item De~Launey and Levin's proofs~\cite[Theorem~4.1]{DL10} give \(N_{n,4t}\sim A_{n,4t}\) as \(t/n^{12}\to\infty\), though this is never explicitly stated.
\item Canfield~\cite{Can} established the same asymptotic as \(t/n^4\to\infty\) in unpublished work.
\end{enumerate}

Our main result is a first-order expansion that extends these to \(\alpha=3\) and reveals the leading correction:

\begin{theorem}\label{thm:main-intro}
There exist \(C_0,c>0\) such that for all sufficiently large \(n\) and \(t\ge C_0 n^3\),
\[
  \frac{N_{n,4t}}{A_{n,4t}}
  =1-\frac{\binom{n}{3}}{8t}
  +O\!\left(\frac{n^2}{t}+\frac{n^{5/2}}{t^{3/2}}+\frac{n^6}{t^2}+e^{-cn^2}\right).
\]
\end{theorem}

\noindent The interesting regime is \(n\to\infty\) with \(t\ge C_0 n^3\), since fixed~\(n\) is already covered by~\cite{DL10}. When \(t/n^3\to\infty\) all corrections vanish, recovering \(N_{n,4t}\sim A_{n,4t}\) (Corollary~\ref{cor:uniform}). When \(t=\Theta n^3\) with \(\Theta\ge C_0\) fixed, the leading correction is \(\binom{n}{3}/(8\Theta n^3)\approx 1/(48\Theta)\); the remaining error is \(O(\Theta^{-2}+1/n+e^{-cn^2})\), so for \(\Theta\) large and \(n\) large the correction dominates the error. Thus the expansion captures the first deviation of \(N_{n,4t}\) from the scale~\(A_{n,4t}\). Below \(\alpha=3\) the asymptotics of \(N_{n,4t}\) are open (Figure~\ref{fig:regimes}).

As in~\cite{DL10}, the proof proceeds by estimation of the Fourier integral~\eqref{eq:fourier-inversion}, split into a \emph{primary} contribution from neighborhoods of the points where \(|\psi|=1\) and a \emph{residual} over the rest of the torus. Beyond improving the exponent, a motivation for pushing the Fourier-analytic approach is that it provides a direct, self-contained route to the counting asymptotics, and it is valuable to understand the limits of these tools. We describe the proof structure below, then compare with the approaches of~\cite{DL10} and~\cite{Can} in Section~\ref{sec:comparison}.

\begin{figure}[t]
\centering
\begin{tikzpicture}
  \def\xone{0}       
  \def\xthree{3}     
  \def\xfour{5.2}    
  \def\xbreakL{7}    
  \def\xbreakR{8}    
  \def\xtwelve{10.2} 
  \def\xend{12.5}    

  \draw[thick] (\xone,0) -- (\xbreakL,0);
  \draw[thick] (\xbreakR,0) -- (\xend,0);
  \draw[-{Stealth},thick] (\xend,0) -- (\xend+0.6,0)
    node[right] {\(\alpha\)};

  \draw[thick] (\xbreakL,0) -- ++(0.25,0.15) -- ++(0.25,-0.15)
    -- ++(0.25,0.15) -- ++(0.25,-0.15);

  \fill[gray!12] (\xone,-0.3) rectangle (\xthree,0.3);
  \draw[pattern=north east lines, pattern color=gray!30, draw=none]
    (\xone,-0.3) rectangle (\xthree,0.3);

  \draw[thick] (\xone,0.13) -- (\xone,-0.13);
  \node[below, font=\small] at (\xone,-0.18) {\(1\)};
  \draw[thick] (\xthree,0.13) -- (\xthree,-0.13);
  \node[below, font=\small] at (\xthree,-0.18) {\(3\)};
  \draw[thick] (\xfour,0.13) -- (\xfour,-0.13);
  \node[below, font=\small] at (\xfour,-0.18) {\(4\)};
  \draw[thick] (\xtwelve,0.13) -- (\xtwelve,-0.13);
  \node[below, font=\small] at (\xtwelve,-0.18) {\(12\)};

  \fill (\xthree,0) circle (2.5pt);
  \fill (\xfour,0) circle (2.5pt);
  \fill (\xtwelve,0) circle (2.5pt);

  \node[above=8pt, font=\footnotesize\bfseries] at (\xthree,0.13)
    {Theorem~\ref{thm:main-intro}};
  \node[above=8pt, font=\footnotesize] at (\xfour,0.13)
    {Canfield};
  \node[above=8pt, font=\footnotesize] at (\xtwelve,0.13)
    {de~Launey--Levin};

  \draw[decorate, decoration={brace, amplitude=5pt, mirror, raise=14pt}]
    (\xthree,-0.3) -- (\xend+0.4,-0.3)
    node[midway, below=20pt, font=\small] {\(N_{n,4t}\sim A_{n,4t}\)};

  \draw[decorate, decoration={brace, amplitude=5pt, mirror, raise=14pt}]
    (\xone,-0.3) -- (\xthree,-0.3)
    node[midway, below=20pt, font=\small] {open};
\end{tikzpicture}
\caption{Each marker indicates the smallest~\(\alpha\) for which the corresponding work establishes \(N_{n,4t}\sim A_{n,4t}\) as \(t/n^\alpha\to\infty\). Corollary~\ref{cor:uniform} extends this to \(\alpha=3\). Theorem~\ref{thm:main-intro} further shows that \(N_{n,4t}/A_{n,4t}\) has a nonvanishing correction when \(t/n^3\) converges to a constant.}\label{fig:regimes}
\end{figure}
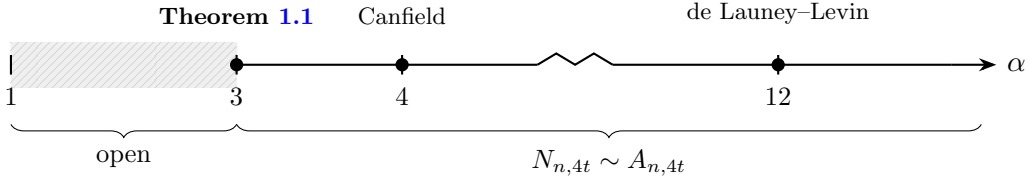

\subsection{Proof overview}\label{sec:proof-overview}

Figure~\ref{fig:composite} illustrates the decomposition. Since partial Hadamard matrices with \(n\ge 3\) rows exist only when the number of columns is divisible by~\(4\), we write the number of columns as~\(4t\) throughout; the integral to evaluate is
\begin{equation}\label{eq:fourier-4t}
  \Prb(S_{4t}=0)
  \;=\;
  (2\pi)^{-d}\int_{\TT^d}\psi(\lambda)^{4t}\,d\lambda.
\end{equation}

\paragraph{Integral decomposition.}
The integrand \(\psi(\lambda)^{4t}\) is largest near the lattice
\[
  \Lambda=\{\lambda\in\TT^d:|\psi(\lambda)|=1\}.
\]
As shown in~\cite[Lemmas~2.2--2.3]{DL10}, each coordinate of every \(\lambda\in\Lambda\) lies in \(\{0,\pm\pi/2,\pi\}\), \(|\Lambda|=2^{2d-n+1}\), and \(\psi(\lambda)^{4t}=1\) for all \(\lambda\in\Lambda\). We place a small box \(\Bbox_\delta(\lambda)=\lambda+[-\delta,\delta]^d\) around each lattice point (Figure~\ref{fig:composite}(a), blue) and call the remainder \(R_\delta=\TT^d\setminus\bigcup_{\lambda\in\Lambda}\Bbox_\delta(\lambda)\) the \emph{residual}.

A multiplicative identity, \(\psi(\lambda+\mu)=\psi(\lambda)\psi(\mu)\) for $\lambda \in \Lambda$ and \(\mu\in[-\pi/4,\pi/4]^d\), together with \(\psi(\lambda)^{4t}=1\), shows that the integral over every primary box equals the integral over the box at the origin. Summing over \(\Lambda\) gives the primary-secondary decomposition~\cite[Proposition~2.1]{DL10}:
\begin{equation}\label{eq:primary-secondary}
  \Prb(S_{4t}=0)
  \;=\;
  \underbrace{|\Lambda|\,(2\pi)^{-d}\int_{\Bbox_\delta}\psi(\mu)^{4t}\,d\mu}_{\text{primary}}
  \;+\;
  \underbrace{(2\pi)^{-d}\int_{R_\delta}\psi(\gamma)^{4t}\,d\gamma}_{\text{residual}}.
\end{equation}
Writing \(K_n:=|\Lambda|(2\pi)^{-d}=2^{2d-n+1}(2\pi)^{-d}\) and splitting the primary integral at the core boundary \(\corereg=\{\|\mu\|^2\le d/t\}\) refines this into a three-term decomposition:
\begin{equation}\label{eq:three-way}
  \Prb(S_{4t}=0)
  =
  \underbrace{K_n\!\int_{\corereg}\psi(\mu)^{4t}\,d\mu}_{\text{core}}
  +\underbrace{K_n\!\int_{\Bbox_\delta\setminus\corereg}\psi(\mu)^{4t}\,d\mu}_{\text{off-core}}
  +\underbrace{(2\pi)^{-d}\!\int_{R_\delta}\psi(\gamma)^{4t}\,d\gamma}_{\text{residual}}.
\end{equation}
Since \(\Prb(S_{4t}=0)\) is real, the proof reduces to three estimates, where \(\Ahat:=A_{n,4t}\cdot 2^{-4nt}\):
\begin{enumerate}[label=(\roman*),nosep]
\item the core contributes \(\bigl(1-\binom{n}{3}/(8t)+\text{lower order}\bigr)\Ahat\) for \(t\ge C_0 n^3\) \hfill(Section~\ref{sec:core}),
\item the off-core contributes \(o(\Ahat)\) as \(t/(n^{8/3}\log t)\to\infty\) \hfill(Section~\ref{sec:residual}),
\item the residual contributes \(o(\Ahat)\) as \(t/(n^{8/3}\log t)\to\infty\) \hfill(Section~\ref{sec:residual}).
\end{enumerate}

\paragraph{Primary term.}
Near the origin, \(\psi\) is close to a Gaussian characteristic function. More precisely, expanding \(\log\psi(\mu)\) to sixth order produces a factorization
\[
  \psi(\mu)^{4t}
  \;=\;
  \underbrace{e^{-2t\|\mu\|^2}}_{\text{Gaussian}}
  \;\cdot\;
  \underbrace{e^{-4itT(\mu)}}_{\text{cubic phase}}
  \;\cdot\;
  (\text{quartic, quintic, and sixth-order terms}),
\]
where \(T(\mu)=\sum_{\text{triangles }\tau}\prod_{e\in\tau}\mu_e\) is a cubic form summing over the \(\binom{n}{3}\) triangles of the complete graph on \(n\) vertices. The Gaussian factor \(e^{-2t\|\mu\|^2}\) concentrates the integrand on the \emph{core} \(\corereg=\{\|\mu\|^2\le d/t\}\) (Figure~\ref{fig:composite}(c), red), where the higher-order terms are small. Outside the core (on the \emph{annulus} and the \emph{corners}, Figure~\ref{fig:composite}(c), yellow and blue), the Gaussian decay makes the contribution exponentially negligible.

The main difficulty on the core is the cubic phase \(\exp(-4itT(\mu))\), which oscillates and does not vanish pointwise. We integrate it directly: because \(T\) is an odd function (\(T(-\mu)=-T(\mu)\)) and \(\corereg\) is centrally symmetric, the imaginary part of \(\int_{\corereg}e^{-2t\|\mu\|^2}e^{-4itT(\mu)}\,d\mu\) vanishes. The loss in the real part is controlled by \(1-\cos(4tT)\le 8t^2T^2\); Gaussian moment estimates on \(T^2\) then give
\[
  \RePart\int_{\corereg}e^{-2t\|\mu\|^2}e^{-4itT(\mu)}\,d\mu
  \;\ge\;
  \bigl[1-O(n^3/t)\bigr]\,\core(d,t),
\]
where \(\core(d,t)=\int_{\corereg}e^{-2t\|\mu\|^2}\,d\mu\). The quartic and quintic terms are controlled in \(L^2\) against the Gaussian weight using hypercontractivity~\cite[Theorem~9.22]{ODo2014} and a vertex-peeling argument; the sixth-order remainder is bounded by Gaussian moment estimates. Altogether, the primary integral equals \(\bigl[1-\binom{n}{3}/(8t)+O(n^2/t+n^{5/2}/t^{3/2}+n^6/t^2+e^{-cn^2})\bigr]\Ahat\) for \(t\ge C_0n^3\). The correction \(\binom{n}{3}/(8t)\) vanishes as \(t/n^3\to\infty\) but is bounded away from zero when \(t\asymp n^3\).

\paragraph{Off-core and residual.}
It remains to show that the off-core integral over \(\Bbox_\delta\setminus\corereg\) and the residual integral over~\(R_\delta\) are both negligible compared to~\(\Ahat\). These regions decompose into pieces with decreasing ease of control (Figure~\ref{fig:composite}(a)--(b)). The key pointwise tool is the following bound from~\cite{DL10}.

\begin{lemma}[Cosine-product bound {\cite[Lemma~3.1]{DL10}}]\label{fact:psi-sq}
For every \(\gamma\in[-\pi,\pi]^d\) and every \(k\in\{1,\dots,n\}\),
\(|\psi(\gamma)|^2\le \frac12+\frac12\prod_{i\ne k}\cos(2\gamma_{\{i,k\}})\).
\end{lemma}

\emph{Odd cells.} The superlattice \(\Lambda_0=\{\lambda\in\TT^d:\text{every coordinate}\in\{0,\pm\pi/2,\pi\}\}\) tiles the torus into quarter-boxes \(\Bbox_{\pi/4}(\lambda)\). All but a \(2^{1-n}\)-fraction of these cells correspond to lattice points in \(\Lambda_0\setminus\Lambda\), where the cosine-product bound (Lemma~\ref{fact:psi-sq}) gives \(|\psi(\gamma)|^2\le 1/2\). On these cells \(|\psi|^{4t}\le 2^{-2t}\), so their total contribution is \(o(\Ahat)\).

\emph{Near shell.} Each even cell \(\Bbox_{\pi/4}(\lambda)\) contains the primary box \(\Bbox_\delta(\lambda)\). Let \(r\in(0,\pi/4)\) be a small absolute constant (fixed in Section~\ref{sec:residual}) such that the cumulant expansion of \(\log\psi\) is valid on the ball \(\Dd_r=\{\lambda:\|\lambda\|^2\le r^2\}\). On the region \(\Dd_r\setminus\Bbox_\delta\) between the primary box and this ball (Figure~\ref{fig:composite}(b), orange), the norm \(\|\mu\|\) exceeds~\(\delta\), and the small-ball decay of~\(\psi\) from the cumulant expansion gives \(|\psi(\mu)|^{4t}\le \exp(-ct\delta^2)\). With \(\delta^2=2d/t\), this contributes at most \(\exp(-cd)\,\Ahat\).

\emph{Far shell.} On the remaining region \(\Bbox_{\pi/4}\setminus\Dd_r\) (Figure~\ref{fig:composite}(b)--(c)), we need a pointwise bound \(|\psi(\lambda)|<1\) strong enough to overwhelm the volume \((\pi/2)^d\). The cosine-product bound (Lemma~\ref{fact:psi-sq}) provides such a contraction when the \emph{row-sums} \(I_k(\lambda):=\sum_{i\ne k}\lambda_{\{i,k\}}^2\) are large for some~\(k\), but it is too weak when all row-sums are small.

To handle this sub-region, we replace the Rademacher entries \(\xi_i\) with i.i.d.\ standard Gaussians \(g_i\) and write \(\psi_G(\lambda)=\E[\exp(i\sum_{i<j}\lambda_{\{i,j\}}g_ig_j)]\) for the resulting characteristic function. The key property is \(|\psi_G(\lambda)|\le(1+2\|\lambda\|^2)^{-1/4}\), which contracts as soon as \(\|\lambda\|\) is bounded away from zero, regardless of how the norm is distributed across coordinates. A Gaussian comparison inequality (Corollary~\ref{cor:weak-comparison}) bounds the distance:
\[
  |\psi(\lambda)-\psi_G(\lambda)|
  \;\le\;
  \tfrac{3}{2}\sum_{k=1}^n I_k(\lambda)^{3/2}.
\]
When all row-sums are small, this error is small and \(|\psi|\) inherits the contraction of~\(|\psi_G|\). When some row-sum \(I_k\) is large, the cosine-product bound \(|\psi(\lambda)|^2\le \frac12+\frac12\prod_{i\ne k}\cos(2\lambda_{\{i,k\}})\) already gives a strict contraction through the exponential decay \(\cos(2x)\le \exp(-8x^2/\pi^2)\). In both cases, \(|\psi|\le q<1\) for a constant \(q\) independent of~\(n\), so the contribution decays exponentially in~\(t\).

Between these two regimes lies an intermediate region where the comparison error is too large for the Gaussian bound but the row-sums are too small for the cosine-product bound to give an \(n\)-independent rate. Here the contraction \(|\psi|^{4t}\le \exp(-ctn^{-2/3})\) is \(n\)-dependent, but it still overwhelms the volume \((\pi/2)^d\) because \(tn^{-2/3}\gg d\log t\asymp n^2\log t\) whenever \(t\gg n^{8/3}\log t\). This is the tightest constraint from the residual, but it is well below the $t \geq C_0 n^3$ constraint required by the primary analysis. The primary term, not the residual, determines the exponent \(\alpha=3\).

\begin{figure}[p]
\centering
\resizebox{\textwidth}{!}{%
\begin{tikzpicture}
  %
  %
  \begin{scope}[xshift=-8.5cm, x=0.76cm, y=0.76cm, every node/.style={font=\footnotesize}]
    \def\bh{0.25}
    \def\dxl{5.2}
    \def\dyl{1.6}
    \def\tdx{0.5}
    \def\tdy{0.5}
    \foreach \sliceIdx/\zlabel/\occA/\occB in {
      3/{\pi}/1.5/3.5,
      2/{\frac{\pi}{2}}/0.5/2.5,
      1/{0}/1.5/3.5,
      0/{-\frac{\pi}{2}}/0.5/2.5} {
      \pgfmathsetmacro{\xoff}{\sliceIdx * \dxl}
      \pgfmathsetmacro{\yoff}{\sliceIdx * \dyl}
      \pgfmathsetmacro{\bdx}{2*\bh*\tdx}
      \pgfmathsetmacro{\bdy}{2*\bh*\tdy}
      \begin{scope}[shift={(\xoff, \yoff)}]
        \fill[gray!12, fill opacity=0.3] (0,4) -- (4,4)
          -- ({4+\tdx},{4+\tdy}) -- (\tdx,{4+\tdy}) -- cycle;
        \fill[pattern=north east lines, pattern color=gray!18]
          (0,4) -- (4,4) -- ({4+\tdx},{4+\tdy}) -- (\tdx,{4+\tdy}) -- cycle;
        \fill[gray!18, fill opacity=0.3] (4,0) -- (4,4)
          -- ({4+\tdx},{4+\tdy}) -- ({4+\tdx},\tdy) -- cycle;
        \fill[pattern=north east lines, pattern color=gray!18]
          (4,0) -- (4,4) -- ({4+\tdx},{4+\tdy}) -- ({4+\tdx},\tdy) -- cycle;
        \fill[gray!8, fill opacity=0.3] (0,0) rectangle (4,4);
        \fill[pattern=north east lines, pattern color=gray!22]
          (0,0) rectangle (4,4);
        %
        \foreach \u in {\occA, \occB} {
          \foreach \v in {\occA, \occB} {
            \fill[white, fill opacity=0.7] ({\u-0.5},{\v+0.5}) -- ({\u+0.5},{\v+0.5})
              -- ({\u+0.5+\tdx},{\v+0.5+\tdy}) -- ({\u-0.5+\tdx},{\v+0.5+\tdy}) -- cycle;
            \fill[orange!25, fill opacity=0.2] ({\u-0.5},{\v+0.5}) -- ({\u+0.5},{\v+0.5})
              -- ({\u+0.5+\tdx},{\v+0.5+\tdy}) -- ({\u-0.5+\tdx},{\v+0.5+\tdy}) -- cycle;
            \fill[white, fill opacity=0.7] ({\u+0.5},{\v-0.5}) -- ({\u+0.5},{\v+0.5})
              -- ({\u+0.5+\tdx},{\v+0.5+\tdy}) -- ({\u+0.5+\tdx},{\v-0.5+\tdy}) -- cycle;
            \fill[orange!35, fill opacity=0.2] ({\u+0.5},{\v-0.5}) -- ({\u+0.5},{\v+0.5})
              -- ({\u+0.5+\tdx},{\v+0.5+\tdy}) -- ({\u+0.5+\tdx},{\v-0.5+\tdy}) -- cycle;
            \fill[white] ({\u-0.5},{\v-0.5}) rectangle ++(1,1);
            \fill[orange!15, fill opacity=0.3] ({\u-0.5},{\v-0.5}) rectangle ++(1,1);
            \draw[orange!60!black, thin] ({\u-0.5},{\v-0.5}) rectangle ++(1,1);
            \draw[orange!60!black, thin]
              ({\u-0.5},{\v+0.5}) -- ({\u-0.5+\tdx},{\v+0.5+\tdy})
              -- ({\u+0.5+\tdx},{\v+0.5+\tdy}) -- ({\u+0.5},{\v+0.5})
              ({\u+0.5+\tdx},{\v+0.5+\tdy}) -- ({\u+0.5+\tdx},{\v-0.5+\tdy}) -- ({\u+0.5},{\v-0.5});
          }
        }
        %
        \draw[thick] (0,0) rectangle (4,4);
        \draw[thick] (0,4) -- (\tdx,{4+\tdy}) -- ({4+\tdx},{4+\tdy}) -- (4,4);
        \draw[thick] (4,0) -- ({4+\tdx},\tdy) -- ({4+\tdx},{4+\tdy});
        %
        \foreach \u in {\occA, \occB} {
          \foreach \v in {\occA, \occB} {
            \fill[blue!60!black!35] ({\u-\bh},{\v+\bh}) -- ({\u+\bh},{\v+\bh})
              -- ({\u+\bh+\bdx},{\v+\bh+\bdy}) -- ({\u-\bh+\bdx},{\v+\bh+\bdy}) -- cycle;
            \fill[blue!60!black!60] ({\u+\bh},{\v-\bh}) -- ({\u+\bh},{\v+\bh})
              -- ({\u+\bh+\bdx},{\v+\bh+\bdy}) -- ({\u+\bh+\bdx},{\v-\bh+\bdy}) -- cycle;
            \fill[blue!60!black!45] ({\u-\bh},{\v-\bh}) rectangle ++(2*\bh, 2*\bh);
            \draw[blue!60!black!85, thin]
              ({\u-\bh},{\v-\bh}) rectangle ++(2*\bh, 2*\bh)
              ({\u-\bh},{\v+\bh}) -- ({\u-\bh+\bdx},{\v+\bh+\bdy})
              -- ({\u+\bh+\bdx},{\v+\bh+\bdy}) -- ({\u+\bh},{\v+\bh})
              ({\u+\bh+\bdx},{\v+\bh+\bdy}) -- ({\u+\bh+\bdx},{\v-\bh+\bdy}) -- ({\u+\bh},{\v-\bh});
            \fill[blue!75!black] (\u,\v) circle (1.3pt);
          }
        }
      \end{scope}
    }
    %
    \pgfmathsetmacro{\dxT}{3*\dxl}
    \pgfmathsetmacro{\dyT}{3*\dyl}
    \draw[gray!35, thin, densely dashed] (0,0) -- (\dxT,\dyT);
    \draw[gray!35, thin, densely dashed] (4,0) -- ({4+\dxT},\dyT);
    \draw[gray!35, thin, densely dashed] (4,4) -- ({4+\dxT},{4+\dyT});
    \draw[gray!35, thin, densely dashed] (0,4) -- (\dxT,{4+\dyT});
    %
    \pgfmathsetmacro{\titleX}{2 + 1.5*\dxl}
    \pgfmathsetmacro{\titleY}{3*\dyl + 4 + 1.0}
    \node[above, font=\normalsize\bfseries] at (\titleX, \titleY)
      {$\TT^3$\quad ($n=3$, $d=3$)};
    %
    \pgfmathsetmacro{\legX}{2 + 1.5*\dxl}
    \node[font=\footnotesize] at (\legX, -1.5) {%
      {\color{blue!75!black}$\bullet$}\;$\Lambda$
      \quad
      {\color{blue!60!black!45}\rule{4pt}{4pt}}\;$\Bbox_\delta(\lambda)$
      \quad
      {\color{orange!25}\rule{4pt}{4pt}}\;$\Bbox_{\pi/4}(\lambda)$
      \quad
      \tikz[baseline=-0.6ex]{\fill[gray!20] (0,0) rectangle (5pt,5pt);
        \fill[pattern=north east lines, pattern color=gray!55] (0,0) rectangle (5pt,5pt);
        \draw[gray!50] (0,0) rectangle (5pt,5pt);}\;$R_\delta^{\mathrm{odd}}$
    };
    %
    \node[font=\footnotesize] at (\legX, -2.3)
      {$R_\delta = \TT^d\!\setminus\!\bigcup_{\lambda\in\Lambda}\Bbox_\delta(\lambda)
        = {\color{orange!70!black}R_\delta^{\mathrm{even}}}
        \cup {\color{gray!50!black}R_\delta^{\mathrm{odd}}}$};
    %
    \node[font=\footnotesize] at (\legX, -3.1) {(a) Torus decomposition};
  \end{scope}
  %
  %
  \begin{scope}[xshift=-4.3cm, yshift=-8.0cm, scale=0.80, every node/.style={font=\footnotesize}]
    \def\bigbox{3.5}
    \def\rr{2.1}
    \def\smbox{1.6}
    \def\dx{1.4}
    \def\dy{1.4}
    \pgfmathsetmacro{\dvx}{\dx/(2*\bigbox)}
    \pgfmathsetmacro{\dvy}{\dy/(2*\bigbox)}
    \pgfmathsetmacro{\pdx}{\dx*\smbox/\bigbox}
    \pgfmathsetmacro{\pdy}{\dy*\smbox/\bigbox}
    %
    \begin{scope}[
      shift={({-\bigbox},{\bigbox})},
      x={(1,0)},
      y={(\dvx,\dvy)}
    ]
      \clip (0,0) -- ({2*\bigbox},0) -- ({2*\bigbox},{2*\bigbox}) -- (0,{2*\bigbox}) -- cycle;
      \fill[orange!18] (0,0) -- ({2*\bigbox},0) -- ({2*\bigbox},{2*\bigbox}) -- (0,{2*\bigbox}) -- cycle;
    \end{scope}
    \draw[thick, orange!60!black] (-\bigbox,\bigbox) -- ({-\bigbox+\dx},{\bigbox+\dy})
      -- ({\bigbox+\dx},{\bigbox+\dy}) -- (\bigbox,\bigbox);
    %
    \begin{scope}[
      shift={({\bigbox},{-\bigbox})},
      x={(\dvx,\dvy)},
      y={(0,1)}
    ]
      \clip (0,0) -- ({2*\bigbox},0) -- ({2*\bigbox},{2*\bigbox}) -- (0,{2*\bigbox}) -- cycle;
      \fill[orange!22] (0,0) -- ({2*\bigbox},0) -- ({2*\bigbox},{2*\bigbox}) -- (0,{2*\bigbox}) -- cycle;
    \end{scope}
    \draw[thick, orange!60!black] (\bigbox,-\bigbox) -- ({\bigbox+\dx},{-\bigbox+\dy})
      -- ({\bigbox+\dx},{\bigbox+\dy});
    %
    \pgfmathsetmacro{\scx}{0.3*\dvx*\bigbox}
    \pgfmathsetmacro{\scy}{0.3*\dvy*\bigbox}
    %
    \begin{scope}
      \clip (-\bigbox,-\bigbox) rectangle (\bigbox,\bigbox);
      \fill[orange!15] (-\bigbox,-\bigbox) rectangle (\bigbox,\bigbox);
      \fill[yellow!25] (\scx,\scy) circle (\rr);
      \fill[white] (-\smbox,-\smbox) rectangle (\smbox,\smbox);
    \end{scope}
    \draw[very thick, orange!60!black] (-\bigbox,-\bigbox) rectangle (\bigbox,\bigbox);
    \begin{scope}
      \clip (-\bigbox,-\bigbox) rectangle (\bigbox,\bigbox);
      \draw[thick, densely dashed, red!60!black] (\scx,\scy) circle (\rr);
      \begin{scope}
        \clip (\scx,\scy) circle (\rr);
        \draw[red!40!black, thin, opacity=0.4] (\scx,\scy) ellipse ({\rr} and {0.35*\rr});
        \draw[red!40!black, thin, opacity=0.4] (\scx,\scy) ellipse ({0.35*\rr} and {\rr});
      \end{scope}
    \end{scope}
    %
    \draw[gray!50, semithick, densely dashed]
      (-\smbox,\smbox) -- ({-\smbox+\pdx},{\smbox+\pdy})
      -- ({\smbox+\pdx},{\smbox+\pdy}) -- (\smbox,\smbox);
    \draw[gray!50, semithick, densely dashed]
      ({\smbox+\pdx},{\smbox+\pdy}) -- ({\smbox+\pdx},{-\smbox+\pdy})
      -- (\smbox,-\smbox);
    \draw[gray!50, semithick, densely dashed] (-\smbox,-\smbox) rectangle (\smbox,\smbox);
    %
    %
    \node[left, yellow!50!black, font=\scriptsize, align=right] at ({-\bigbox-0.6}, 0.8)
      {Near shell\\[-1pt]{\scriptsize $\Dd_r\!\setminus\!\Bbox_\delta$}};
    \draw[-{Stealth[length=3pt]}, yellow!50!black, thick]
      ({-\bigbox-0.55}, 0.8) -- ({-\smbox-0.15}, 0.8);
    \node[left, orange!60!black, font=\scriptsize, align=right] at ({-\bigbox-0.6}, -2.2)
      {Far shell\\[-1pt]{\scriptsize $\Bbox_{\pi/4}\!\setminus\!\Dd_r$}};
    \draw[-{Stealth[length=3pt]}, orange!60!black, thick]
      ({-\bigbox-0.55}, -2.2) -- ({-\bigbox+0.05}, -2.2);
    %
    \node[above, font=\small] at ({\dx/2}, {\bigbox+\dy+0.5})
      {$\Bbox_{\pi/4}(\lambda)$};
    %
    \node[font=\footnotesize] at (0, {-\bigbox-1.2}) {(b) Even cell};
    %
    \coordinate (bd-tr) at (\smbox, \smbox);
    \coordinate (bd-br) at (\smbox, -\smbox);
  \end{scope}
  %
  %
  \begin{scope}[xshift=4.3cm, yshift=-8.0cm, scale=0.80, every node/.style={font=\footnotesize}]
    \def\bs{3.0}
    \def\rr{4.0}
    \def\rt{1.2}
    \def\dx{1.6}
    \def\dy{1.6}
    \pgfmathsetmacro{\dvx}{\dx/(2*\bs)}
    \pgfmathsetmacro{\dvy}{\dy/(2*\bs)}
    \pgfmathsetmacro{\rint}{sqrt(\rr*\rr - \bs*\bs)}
    %
    \begin{scope}[
      shift={({-\bs},{\bs})},
      x={(1,0)},
      y={(\dvx,\dvy)}
    ]
      \clip (0,0) -- ({2*\bs},0) -- ({2*\bs},{2*\bs}) -- (0,{2*\bs}) -- cycle;
      \fill[blue!60!black!35] (0,0) rectangle ({2*\bs},{2*\bs});
    \end{scope}
    \draw[thick] (-\bs,\bs) -- ({-\bs+\dx},{\bs+\dy})
      -- ({\bs+\dx},{\bs+\dy}) -- (\bs,\bs);
    %
    \begin{scope}[
      shift={({\bs},{-\bs})},
      x={(\dvx,\dvy)},
      y={(0,1)}
    ]
      \clip (0,0) -- ({2*\bs},0) -- ({2*\bs},{2*\bs}) -- (0,{2*\bs}) -- cycle;
      \fill[blue!60!black!50] (0,0) rectangle ({2*\bs},{2*\bs});
    \end{scope}
    \draw[thick] (\bs,-\bs) -- ({\bs+\dx},{-\bs+\dy})
      -- ({\bs+\dx},{\bs+\dy});
    %
    \pgfmathsetmacro{\cx}{\dvx*\bs}
    \pgfmathsetmacro{\cy}{\dvy*\bs}
    %
    \begin{scope}
      \clip (-\bs,-\bs) rectangle (\bs,\bs);
      \fill[blue!60!black!45] (-\bs,-\bs) rectangle (\bs,\bs);
      \fill[yellow!40] (\cx,\cy) circle (\rr);
      \fill[red!35] (\cx,\cy) circle (\rt);
    \end{scope}
    \draw[very thick] (-\bs,-\bs) rectangle (\bs,\bs);
    \draw[thick, yellow!50!black, densely dashed] (\cx,\cy) circle (\rr);
    \begin{scope}
      \clip (\cx,\cy) circle (\rr);
      \draw[yellow!40!black, thin, opacity=0.4] (\cx,\cy) ellipse ({\rr} and {0.35*\rr});
      \draw[yellow!40!black, thin, opacity=0.4] (\cx,\cy) ellipse ({0.35*\rr} and {\rr});
    \end{scope}
    \draw[thick, red!60!black] (\cx,\cy) circle (\rt);
    \begin{scope}
      \clip (\cx,\cy) circle (\rt);
      \draw[red!40!black, thin, opacity=0.4] (\cx,\cy) ellipse ({\rt} and {0.35*\rt});
      \draw[red!40!black, thin, opacity=0.4] (\cx,\cy) ellipse ({0.35*\rt} and {\rt});
    \end{scope}
    \fill[red!80!black] (\cx,\cy) circle (2pt);
    \node[below, font=\scriptsize] at (\cx,{\cy-0.15}) {$\lambda$};
    \node[red!60!black, font=\scriptsize\bfseries] at (\cx,{\cy+0.55}) {Core $\Dd_{\mathrm{core}}$};
    %
    \node[font=\scriptsize, yellow!50!black] at (\cx, {-\bs+1.0}) {Annulus};
    \node[font=\scriptsize, yellow!50!black] at (\cx, {-\bs+0.6})
      {$(\Bbox_\delta\!\cap\!\Dd_r)\!\setminus\!\Dd_{\mathrm{core}}$};
    \node[left, blue!60!black, font=\scriptsize, align=right] at ({-\bs-0.6}, -2.2)
      {Corners\\[-1pt]{\scriptsize $\Bbox_\delta\!\setminus\!\Dd_r$}};
    \draw[-{Stealth[length=3pt]}, blue!60!black, thick]
      ({-\bs-0.55}, -2.2) -- ({-\bs+0.05}, -2.2);
    %
    \node[above, font=\small] at ({\dx/2}, {3.5+1.4+0.5})
      {$\Bbox_\delta(\lambda)$};
    %
    \node[font=\footnotesize] at (0, {-3.5-1.2}) {(c) Primary box};
    %
    \coordinate (pr-tl) at (-\bs, \bs);
    \coordinate (pr-bl) at (-\bs, -\bs);
  \end{scope}
  %
  %
  \draw[-{Stealth[length=5pt]}, thick, blue!60!black]
    (bd-tr) -- (pr-tl);
  \draw[-{Stealth[length=5pt]}, thick, blue!60!black]
    (bd-br) -- (pr-bl);
\end{tikzpicture}}%
\caption{Composite overview.
(a)~The torus $\TT^3$ decomposes into primary boxes $\Bbox_\delta$
(blue), even residual cells (orange), and odd cells (hatched).
(b)~A single even cell $\Bbox_{\pi/4}$; the blue $\Bbox_\delta$
is expanded in~(c).
(c)~Inside $\Bbox_\delta$: the core $\Dd_{\mathrm{core}}$ (red),
annulus (yellow), and corners (blue).}
\label{fig:composite}
\end{figure}
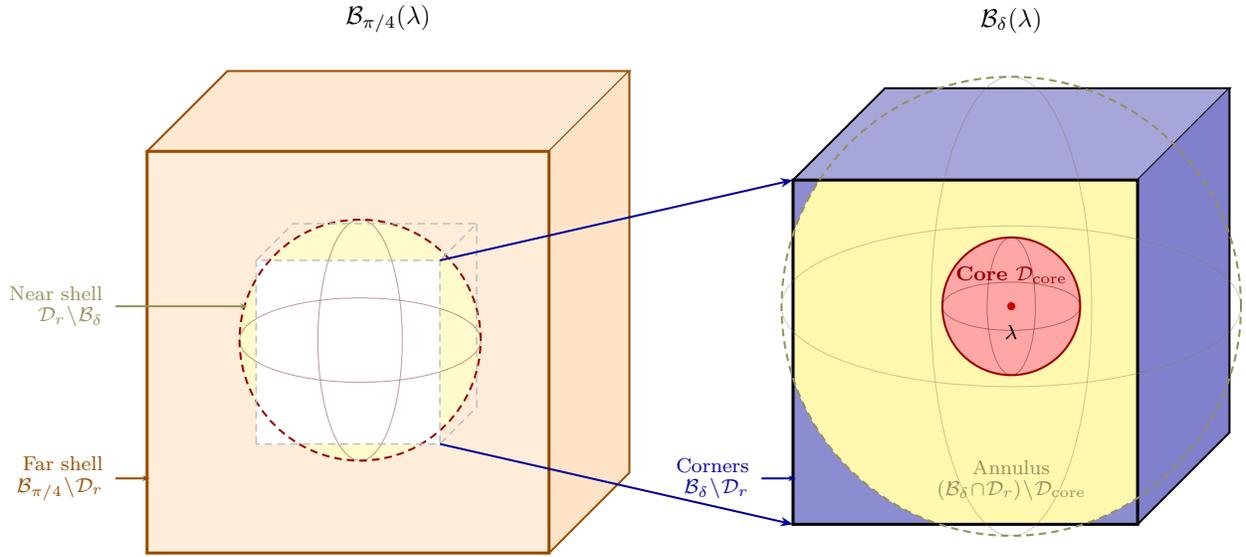

\subsection{Comparison with prior work}\label{sec:comparison}

All three proofs (\cite{DL10}, \cite{Can}, and the present work) use the primary/residual decomposition~\eqref{eq:primary-secondary}. The exponent~\(\alpha\) is determined by how two difficulties are resolved: the cubic phase on the primary boxes, and the pointwise decay on the far shell. We describe each approach in the notation introduced above.

\paragraph{The cubic phase and the primary integral.}
On the primary box \(\Bbox_\delta\), the characteristic function \(\psi\) is close to Gaussian but has a non-trivial cubic phase driven by the triangle form~\(T\). The three proofs handle this phase differently.

De~Launey and Levin~\cite{DL10} separate \(\psi\) into real and imaginary parts. On \(\Bbox_\delta\) the ratio \(\beta:=\operatorname{Im}(\psi)/\operatorname{Re}(\psi)\) is at most \(C(n\delta)^3\), and they bound \(\operatorname{Re}(\psi^{4t})\) below by \(\operatorname{Re}(\psi)^{4t}\). This lower bound requires \(4t\beta<1\): if \(\beta\) is too large relative to \(1/t\), raising \(\psi\) to the \(4t\)-th power amplifies the imaginary part and destroys the bound. Since \(\beta\le C(n\delta)^3\), the constraint forces \(\delta\le C(tn^3)^{-1/3}\), shrinking the primary box as \(t\) and \(n\) grow. However, shrinking~\(\delta\) enlarges the residual region, which is harder to control (see below).

Canfield~\cite{Can} avoids this hard constraint by integrating one coordinate at a time. The cubic phase can be written \(T(\mu)=\sum_{j<k}\mu_{jk}B_{jk}\), where \(B_{jk}:=\sum_\ell\mu_{j\ell}\mu_{k\ell}\) sums over the remaining edges of triangles through~\(\{j,k\}\). Because \(\mu_{jk}\) couples linearly to~\(B_{jk}\), completing the square \(\exp(-t\mu_{jk}^2/2-it\mu_{jk}B_{jk})=\exp(-tB_{jk}^2/2)\cdot\exp(-t(\mu_{jk}+iB_{jk})^2/2)\) reduces each one-dimensional integral to a Gaussian at the cost of a factor \(\exp(-tB_{jk}^2/2)\). After all \(d\) steps, the accumulated error is \(d\cdot O(n^2/t)=O(n^4/t)\), giving \(\alpha=4\). Canfield identifies \(\alpha=3\) as the natural target but leaves the required cubic integration as an open problem (see~\cite{Can}).

Our approach integrates the cubic phase in one step rather than absorbing it variable by variable. As described in Section~\ref{sec:proof-overview}, confining the analysis to the core~\(\corereg\) and exploiting the antisymmetry of~\(T\) gives an error of \(O(n^3/t)\), one power of~\(n\) better than Canfield's accumulated error. Because the primary lower bound lives entirely on~\(\corereg\), the remainder of~\(\Bbox_\delta\) is bounded as part of the off-core estimates~\eqref{eq:three-way}, and the choice of~\(\delta\) is decoupled from the primary analysis.

\paragraph{The residual and the far shell.}
De~Launey and Levin bound the entire residual~\(R_\delta\) by the uniform estimate \(|\psi(\gamma)|\le\cos\delta\le \exp(-(11/24)\delta^2)\), yielding \(|(2\pi)^{-d}\!\int_{R_\delta}\!\psi^{4t}|\le \exp(-(11/6)t\delta^2)\). For this to be \(o(\Ahat)\) one needs \(t\delta^2\gg d\log t\). Combined with the constraint \(\delta\le C(tn^3)^{-1/3}\) from their primary analysis, this gives \(t^{1/3}/n^2\gg n^2\log t\), hence \(t\gg n^{12}\). Canfield improves the residual bound by a factor of~\(n\) in the exponent, but this is not his bottleneck.

In our proof, the odd cells and the near shell \(\Dd_r\setminus\Bbox_\delta\) are negligible by elementary bounds (Section~\ref{sec:proof-overview}). The critical piece is the far shell \(\Bbox_{\pi/4}\setminus\Dd_r\). When all row-sums \(I_k\) are small, the cosine-product bound (Lemma~\ref{fact:psi-sq}) decays at an \(n\)-dependent rate that is too weak; the Gaussian comparison inequality (Corollary~\ref{cor:weak-comparison}) provides the \(n\)-independent contraction needed in this sub-region.

\medskip

The progression from \(\alpha=12\) to \(\alpha=3\) is driven by a decoupling of the primary and residual estimates. De~Launey and Levin's two-way split forces \(\delta\le C(tn^3)^{-1/3}\) on the primary side while the residual needs \(t\delta^2\) large. Canfield's coordinate-by-coordinate integration frees the residual from~\(\delta\) but accumulates \(O(n^4/t)\) error. The three-way decomposition removes this coupling: the Gaussian comparison resolves the far-shell bottleneck, and all thresholds align at \(t\asymp n^3\). Pushing below \(\alpha=3\) would require finer control of the cubic phase on the core, since the residual is already negligible for \(t\gg n^{8/3}\log t\).

\subsection*{Acknowledgement of AI tools}

The author used AI tools throughout this project for literature search, editing, and figure generation. More substantially, the author built a custom harness around GPT 5.4 Pro that identified bottlenecks in the existing proof approaches and, after considerable iteration, helped guide the analysis to \(\alpha=3\). Details of this process and harness will be discussed elsewhere.

\subsection{Notation and preliminaries}\label{sec:prelim}

We fix notation and collect the preliminary results used in Sections~\ref{sec:core}--\ref{sec:counting}. We include short proofs of some of the results from~\cite{DL10} for completeness.

\paragraph{Notation.}
We retain \(n\), \(d=\binom{n}{2}\), \(Z(y)\), \(\xi\), \(\psi\), \(N_{n,s}\), \(\Lambda\), \(\Bbox_\delta\), \(R_\delta\), \(\Ahat\), and \(A_{n,4t}\) as defined in Section~\ref{sec:proof-overview}. For i.i.d.\ copies \(\xi^{(1)},\dots,\xi^{(s)}\) of~\(\xi\), write \(S_s=\sum_{r=1}^s Z(\xi^{(r)})\).

For the cumulant expansion we write \(X_\lambda=\lambda\cdot Z(\xi)\), \(s(\lambda)=\|\lambda\|^2\), \(\kappa_r(\lambda)\) for the \(r\)-th cumulant of \(X_\lambda\), \(T(\lambda)=\kappa_3(\lambda)/6\), \(Q(\lambda)=\kappa_4(\lambda)/24\), and \(P(\lambda)=\kappa_5(\lambda)/120\). For \(r>0\), set \(\Dd_r=\{\lambda\in\R^d:\|\lambda\|^2\le r^2\}\). The core region, the core Gaussian mass, and the full Gaussian integral are
\[
  \corereg=\left\{\lambda\in\R^d:\|\lambda\|^2\le \frac{d}{t}\right\},
  \qquad
  \core(d,t)=\int_{\corereg}e^{-2t\|\lambda\|^2}\,d\lambda,
  \qquad
  F(d,t)=\left(\frac{\pi}{2t}\right)^{d/2}.
\]

\paragraph{Results from~\cite{DL10}.}
The lattice \(\Lambda\) has a rigid algebraic structure, which we use to decompose the integral.

\begin{lemma}[Lattice structure]\label{fact:lambda-facts}
The following hold.
\begin{enumerate}[label=(\roman*),nosep]
\item If \(\lambda\in\Lambda\) and \(\gamma\in\Bbox_{\pi/4}\), then
\(\psi(\lambda+\gamma)=\psi(\lambda)\psi(\gamma)\).
\item Every coordinate of every \(\lambda\in\Lambda\) lies in \(\{0,\pm\pi/2,\pi\}\).
\item If \(\delta<\pi/4\), then the boxes \(\Bbox_\delta(\lambda)\), \(\lambda\in\Lambda\), are pairwise disjoint.
\item The multiset \(\{\psi(\lambda):\lambda\in\Lambda\}\) consists of \(\pm1,\pm i\),
each with multiplicity \(2^{2d-n-1}\). Consequently \(\psi(\lambda)^{4t}=1\) for
\(\lambda\in\Lambda\), and \(|\Lambda|=2^{2d-n+1}\).
\item Each \(\lambda\in\Lambda_0\) decomposes uniquely as \(\lambda^{(1)}+\lambda^{(2)}\) with \(\lambda^{(1)}\in\{0,\pi\}^d\) and \(\lambda^{(2)}\in\{0,\pi/2\}^d\). Associate to \(\lambda^{(2)}\) the graph \(G_{\lambda^{(2)}}\) on \([n]\) with edge \(\{i,j\}\) whenever \(\lambda^{(2)}_{\{i,j\}}=\pi/2\). Then \(\lambda\in\Lambda\) if and only if every vertex of \(G_{\lambda^{(2)}}\) has even degree.
\end{enumerate}
The first two parts are \cite[Lemma~2.2]{DL10}, part~(iii) follows immediately from part~(ii), part~(iv) is \cite[Lemma~2.3]{DL10}, and part~(v) is \cite[Lemma~2.4]{DL10}.
\end{lemma}

The multiplicative identity~(i) and the fact that every \(\psi(\lambda)^4=1\) for \(\lambda\in\Lambda\)~(iv) justify the decomposition~\eqref{eq:primary-secondary} from the proof overview.

\begin{proposition}[Primary-secondary decomposition {\cite[Proposition~2.1]{DL10}}]\label{prop:primary-secondary}
For every \(t\ge 1\) and every \(0<\delta<\pi/4\),
\[
  \Prb(S_{4t}=0)
  =
  2^{2d-n+1}(2\pi)^{-d}\int_{\Bbox_\delta}\psi(\lambda)^{4t}\,d\lambda
  +(2\pi)^{-d}\int_{R_\delta}\psi(\gamma)^{4t}\,d\gamma.
\]
\end{proposition}

The residual \(R_\delta\) further splits into two types of cells, distinguished by the parity of the underlying lattice point. The odd cells are negligible by an elementary sign argument; the even cells require more delicate bounds.

\begin{lemma}[Residual decomposition {\cite[Lemma~2.7]{DL10}}]\label{lem:residual-decomp}
For \(0<\delta<\pi/4\),
\[
  R_\delta \;=\; R_\delta^{\mathrm{even}} \cup R_\delta^{\mathrm{odd}},
  \qquad
    R_\delta^{\mathrm{even}}
  = \bigcup_{\lambda\in\Lambda}\bigl(\lambda+[\Bbox_{\pi/4}\setminus\Bbox_\delta]\bigr),
  \qquad
    R_\delta^{\mathrm{odd}}
  = R_\delta\setminus R_\delta^{\mathrm{even}},
\]
and the sets in the union are disjoint.
\end{lemma}

\begin{proof}
Since every coordinate of every \(\lambda\in\Lambda\) lies in \(\{0,\pm\pi/2,\pi\}\) (Lemma~\ref{fact:lambda-facts}(ii)), the translates \(\Bbox_{\pi/4}(\lambda)\), \(\lambda\in\Lambda_0\), tile \(\TT^d\), where \(\Lambda_0\) is the set of all points with coordinates in \(\{0,\pm\pi/2,\pi\}\). The lattice \(\Lambda\) is a subset of~\(\Lambda_0\). Each tile centered at a point in~\(\Lambda\) contributes \(\Bbox_{\pi/4}(\lambda)\setminus\Bbox_\delta(\lambda)\) to \(R_\delta^{\mathrm{even}}\), and each tile centered at a point in \(\Lambda_0\setminus\Lambda\) contributes a full quarter-box to \(R_\delta^{\mathrm{odd}}\).
\end{proof}

On odd cells the characteristic function is bounded by a constant less than one, making the contribution of these cells exponentially small. The following bound is implicit in the proof of~\cite[Proposition~3.2]{DL10}; we give a self-contained argument.

\begin{lemma}[Odd-cell bound]\label{lem:odd-bound}
If \(\gamma\in R_\delta^{\mathrm{odd}}\), then \(|\psi(\gamma)|^2\le \tfrac12\).
\end{lemma}

\begin{proof}
It suffices to show that some \(k\in\{1,\dots,n\}\) satisfies \(\prod_{i\ne k}\cos(2\gamma_{\{i,k\}})\le 0\), since the cosine-product bound (Lemma~\ref{fact:psi-sq}) then gives \(|\psi(\gamma)|^2\le \tfrac12\).

By Lemma~\ref{lem:residual-decomp}, \(\gamma\) lies in a quarter-box
\(\Bbox_{\pi/4}(\lambda)\) with \(\lambda\in\Lambda_0\setminus\Lambda\).
Write \(\lambda=\lambda^{(1)}+\lambda^{(2)}\),
where each coordinate of \(\lambda^{(1)}\) lies in \(\{0,\pi\}\) and each coordinate of
\(\lambda^{(2)}\) lies in \(\{0,\pi/2\}\). By Lemma~\ref{fact:lambda-facts}(v), \(\lambda\in\Lambda\) if and only if every vertex of \(G_{\lambda^{(2)}}\) has even degree. Since \(\lambda\notin\Lambda\), some vertex \(k\) has odd degree, i.e., an odd number of the coordinates \(\lambda^{(2)}_{\{i,k\}}\), \(i\ne k\), equal \(\pi/2\).

Write \(\gamma=\lambda^{(1)}+\lambda^{(2)}+\mu\) with \(\mu\in \Bbox_{\pi/4}\). For each \(i\ne k\),
\[
  \cos(2\gamma_{\{i,k\}})
  =
  \cos\!\bigl(2\lambda^{(1)}_{\{i,k\}}+2\lambda^{(2)}_{\{i,k\}}+2\mu_{\{i,k\}}\bigr).
\]
Since \(2\lambda^{(1)}_{\{i,k\}}\in\{0,2\pi\}\) and
\(2\lambda^{(2)}_{\{i,k\}}\in\{0,\pi\}\), this equals
\((-1)^{b_{ik}}\cos(2\mu_{\{i,k\}})\) with
\(b_{ik}\in\{0,1\}\), and an odd number of the \(b_{ik}\) equal \(1\).
Because \(|\mu_{\{i,k\}}|\le \pi/4\), each factor \(\cos(2\mu_{\{i,k\}})\ge 0\),
so the product \(\prod_{i\ne k}\cos(2\gamma_{\{i,k\}})\) carries an odd number of sign flips and is therefore nonpositive.
\end{proof}

The sixth-order expansion of \(\log\psi\) (Lemma~\ref{lem:sixth}) requires that \(\RePart\psi\) stay bounded away from zero along a path, so that the principal branch of the logarithm is well-defined. The following bound, which is immediate from~\cite[Lemma~3.1, eq.~(3.2)]{DL10}, provides this.

\begin{lemma}[Positivity of $\RePart\psi$ {\cite[Lemma~3.1]{DL10}}]\label{lem:realpart}
If \(s(\lambda)\le 1/2\), then \(\RePart\psi(\lambda)\in[3/4,1]\).
\end{lemma}

\begin{proof}
Since \(\RePart\psi(\lambda)=\E[\cos X_\lambda]\), the bound
\(\cos x\ge 1-x^2/2\) gives
\(\RePart\psi(\lambda)\ge 1-\frac12 s(\lambda)\ge 3/4\).
\end{proof}

Our core and off-core analysis (Sections~\ref{sec:core}--\ref{sec:residual}) yields asymptotics as \(n\to\infty\). For bounded~\(n\), the following fixed-\(n\) result from~\cite{DL10} closes the argument in Section~\ref{sec:counting}.

\begin{lemma}[Fixed-\(n\) asymptotics {\cite[Theorem~4.1]{DL10}}]\label{fact:fixed-n-dl10}
For every fixed integer \(n\ge 2\),
\(N_{n,4t}=[1+o_{t\to\infty}(1)]\,A_{n,4t}\) as \(t\to\infty\).
\end{lemma}

\paragraph{Standard analytic tools.}
The cumulant expansion on the core produces polynomial corrections to the Gaussian integrand. We control these using three standard estimates.

\begin{lemma}[Hypercontractive inequality {\cite[Theorems~9.22 and~11.23]{ODo2014}}]\label{lem:hc}
Let \(H\) be a polynomial of degree at most \(q\) in either independent Rademacher signs or independent centered Gaussians. Then for every \(p\ge 2\),
\(\|H\|_{L^p}\le (p-1)^{q/2}\|H\|_{L^2}\).
\end{lemma}

Hypercontractivity converts \(L^p\) norms of \(Q\) and \(P\) into \(L^2\) norms, which are explicit. The following moment bound controls the sixth-order remainder against the Gaussian weight.

\begin{lemma}[Gaussian radial moments]\label{lem:gaussian-radial}
For every integer \(m\ge 0\) there is a constant \(C_m>0\) such that
\(\int_{\R^d}\|\lambda\|^{2m}e^{-2t\|\lambda\|^2}\,d\lambda
\le C_m(d/t)^m F(d,t)\) for all \(d,t\ge 1\).
\end{lemma}

The following identity is used in two ways: the real case drives is used in vertex-peeling argument in the quartic bounds (Proposition~\ref{prop:quartic-bounds}), and the complex case gives a determinant formula for the Gaussian characteristic function \(\psi_G\) in the far-shell decomposition (Proposition~\ref{prop:far-shell}).

\begin{lemma}[Gaussian quadratic integral]\label{lem:gaussian-quadratic}
Let \(A\) be a complex symmetric \(m\times m\) matrix whose real part is positive definite. Then
\(\int_{\R^m}e^{-x^\top A x}\,dx=\pi^{m/2}\det(A)^{-1/2}\).
In particular, if \(g\sim N(0,I_m)\) and \(M\) is real symmetric, taking \(A=\tfrac12(I-iM)\) gives \(\E[e^{ig^\top Mg/2}]=\det(I-iM)^{-1/2}\) and \(|\E[e^{ig^\top Mg/2}]|=\det(I+M^2)^{-1/4}\).
\end{lemma}

\begin{proof}
Write \(A=\RePart(A)+i\,\ImPart(A)\). The substitution \(y=\RePart(A)^{1/2}x\) gives
\[
  \int_{\R^m}e^{-x^\top Ax}\,dx
  =\det(\RePart A)^{-1/2}\int_{\R^m}e^{-\|y\|^2-iy^\top Sy}\,dy,
\]
where \(S=\RePart(A)^{-1/2}\ImPart(A)\,\RePart(A)^{-1/2}\) is real symmetric.
Diagonalizing \(S=O^\top\!\operatorname{diag}(s_1,\dots,s_m)O\) with \(O\) orthogonal and setting \(z=Oy\), the integral factors as
\[
  \prod_{j=1}^m\int_\R e^{-(1+is_j)z_j^2}\,dz_j
  =\prod_{j=1}^m\sqrt{\frac{\pi}{1+is_j}}
  =\pi^{m/2}\det(I+iS)^{-1/2}.
\]
Since \(A=\RePart(A)^{1/2}(I+iS)\,\RePart(A)^{1/2}\), we have \(\det(A)=\det(\RePart A)\det(I+iS)\), and the result follows.
\end{proof}

\paragraph{Auxiliary lemma.}
The following estimate shows that the Gaussian integral over the core \(\corereg\) captures nearly all of the full Gaussian integral \(F(d,t)=\int_{\R^d}e^{-2t\|\lambda\|^2}\,d\lambda=(\pi/(2t))^{d/2}\).

\begin{lemma}[Core Gaussian mass]\label{lem:inner-core}
There exists an absolute constant \(c_*>0\) such that
\(\core(d,t)\ge (1-e^{-c_*d})F(d,t)\).
\end{lemma}

\begin{proof}
On the region \(\|\lambda\|^2>d/t\), one has \(t\|\lambda\|^2>d\), hence
\(e^{-2t\|\lambda\|^2}\le e^{-d}e^{-t\|\lambda\|^2}\). Therefore
\[
  \core(d,t)
  =F(d,t)-\int_{\|\lambda\|^2>d/t}e^{-2t\|\lambda\|^2}\,d\lambda
  \ge F(d,t)-e^{-d}\int_{\R^d}e^{-t\|\lambda\|^2}\,d\lambda
  =(1-e^{-c_*d})F(d,t),
\]
where \(c_*=1-\frac12\log 2>0\).
\end{proof}

\section{Core}\label{sec:core}

This section establishes the primary estimate: Proposition~\ref{prop:primary-box} gives a first-order expansion of \(K_n\RePart\int_{\corereg}\psi^{4t}\) with leading correction \(-\binom{n}{3}/(8t)\) for \(t\ge C_0n^3\) and \(n\) large. Combined with the off-core upper bounds in Section~\ref{sec:residual}, this proves the main counting theorem.

The argument expands \(\log\psi\) to fifth order on \(\corereg\), isolates the cubic phase \(T(\lambda)\), and shows that \(Q\), \(P\), and \(E_6\) are negligible when \(t\gg n^3\). The imaginary part of the cubic phase vanishes by antisymmetry of~\(T\), and the real part differs from the Gaussian integral \(\core(d,t)\) by \(O(n^3/t)\).

\subsection{Cumulants for the log expansion}

We compute the low-order cumulants of \(X_\lambda\) entering the Taylor expansion of \(\log\psi(\lambda)=\log \E[e^{iX_\lambda}]\). The identities below produce the explicit cubic, quartic, and quintic phase terms \(T\), \(Q\), and \(P\).

The first identity is implicit in the proof of~\cite[Lemma~3.1]{DL10}; we state it separately for later reference.

\begin{lemma}[Triangle formula]\label{lem:triangle}
For every \(\lambda\in\R^d\),
\(T(\lambda)=\sum_{1\le i<j<k\le n}\lambda_{ij}\lambda_{ik}\lambda_{jk}\).
\end{lemma}

\begin{proof}
Expanding the cube gives
\[
  \E[X_\lambda^3]
  =\sum_{e_1,e_2,e_3}\lambda_{e_1}\lambda_{e_2}\lambda_{e_3}
    \E[\chi_{e_1}\chi_{e_2}\chi_{e_3}],
\]
where \(\chi_{\{i,j\}}=\xi_i\xi_j\). The expectation vanishes unless every vertex index appears with even multiplicity. With three edges and no loops, the only such multigraph is a triangle on three distinct vertices. Thus the only surviving ordered triples are the \(3!\) orderings of $(\{i,j\},\{i,k\},\{j,k\})$ for $1\le i<j<k\le n.$
For each such ordering, $(\xi_i\xi_j)(\xi_i\xi_k)(\xi_j\xi_k)=1.$ 
Summing over all triangles and their \(3!\) orderings gives
\[
  \E[X_\lambda^3]
  =6\sum_{1\le i<j<k\le n}\lambda_{ij}\lambda_{ik}\lambda_{jk}.
\]
Since \(T(\lambda)=\frac16\E[X_\lambda^3]\) by definition, the claim follows.
\end{proof}

Define
\[
  Q(\lambda):=\frac{\kappa_4(\lambda)}{24}.
\]
The next lemma gives an explicit formula for~\(Q\). Write
\begin{equation}\label{eq:C4}
  \mathcal{C}_4(\lambda)
  =
    \sum_{\substack{i_1,i_2,i_3,i_4\\ \text{distinct}}}
      \lambda_{\{i_1,i_2\}}
      \lambda_{\{i_2,i_3\}}
      \lambda_{\{i_3,i_4\}}
    \lambda_{\{i_4,i_1\}}
\end{equation}
for the sum over ordered \(4\)-cycles.

\begin{lemma}[Fourth cumulant]\label{lem:ordered-cycle}
\begin{equation}\label{eq:Q}
  Q(\lambda)
  =
  -\frac1{12}\sum_e\lambda_e^4+\frac18\mathcal{C}_4(\lambda).
\end{equation}
\end{lemma}

\begin{proof}
Since \(\kappa_4(\lambda)=\E[X_\lambda^4]-3s(\lambda)^2\), we expand \(\E[X_\lambda^4]\) by writing four copies of \(X_\lambda=\sum_e\lambda_e\chi_e\) and keeping only terms where every vertex has even multiplicity:
\begin{enumerate}[label=(\roman*),nosep]
\item All four edges equal: contributes \(\sum_e\lambda_e^4\).
\item Two distinct edges, each chosen twice: \(\binom42=6\) orderings, contributing \(6\sum_{e<f}\lambda_e^2\lambda_f^2\).
\item Four distinct edges forming an undirected \(4\)-cycle: each cycle monomial \(\lambda_{ab}\lambda_{bc}\lambda_{cd}\lambda_{da}\) appears \(4!=24\) times in the raw expansion and \(8\) times in \(\mathcal{C}_4\) (two cyclic orientations, four starting points), giving coefficient \(24/8=3\).
\end{enumerate}
Therefore
\(\E[X_\lambda^4]
  =\sum_e\lambda_e^4
    +6\sum_{e<f}\lambda_e^2\lambda_f^2
    +3\mathcal{C}_4(\lambda)\).
Since \(3s(\lambda)^2=3\sum_e\lambda_e^4+6\sum_{e<f}\lambda_e^2\lambda_f^2\), subtracting gives \(\kappa_4(\lambda)=-2\sum_e\lambda_e^4+3\mathcal{C}_4(\lambda)\), and dividing by \(24\) yields~\eqref{eq:Q}.
\end{proof}

The formulas~\eqref{eq:C4}--\eqref{eq:Q} express \(\mathcal{C}_4\) and \(Q\) as polynomials in the entries of any symmetric zero-diagonal matrix; we write \(Q_m\) and \(\mathcal{C}_{4,m}\) when the matrix is \(m\times m\). The following lemma is the main tool for the vertex-peeling argument that controls \(Q\).

\begin{lemma}[Peeling identity]\label{lem:peeling}
Let \(A\) be a symmetric zero-diagonal \(n\times n\) matrix with entries \(A_{ij}=\lambda_{\{i,j\}}\). Write \(B\) for the \((n{-}1)\times(n{-}1)\) principal submatrix obtained by deleting row and column \(n\), write \(x\in\R^{n-1}\) for the last column, and set \(M(B)=B^2-\diag(B^2)\). Then
\[
  Q_n(A)
  =
  Q_{n-1}(B)
  +\frac12 x^\top M(B)x
  -\frac1{12}\sum_{a=1}^{n-1}x_a^4.
\]
\end{lemma}

\begin{proof}
The diagonal quartic part splits as
\[
  \sum_{1\le i<j\le n}A_{ij}^4
  =
  \sum_{1\le a<b\le n-1}B_{ab}^4+\sum_{a=1}^{n-1}x_a^4.
\]
The cycles avoiding vertex \(n\) contribute exactly \(\mathcal{C}_{4,n-1}(B)\). Since \(M(B)=B^2-\diag(B^2)\), for \(a\ne c\) one has \(M(B)_{ac}=(B^2)_{ac}=\sum_{b}B_{ab}B_{bc}\), so \(x^\top M(B)x=\sum_{a\ne c}\sum_{b}x_aB_{ab}B_{bc}x_c\).
Because \(B\) has zero diagonal, only triples with \(a,b,c\) distinct contribute. Each monomial \(x_aB_{ab}B_{bc}x_c\) corresponds to the undirected \(4\)-cycle \(n\text{-}a\text{-}b\text{-}c\text{-}n\), which contributes twice to \(x^\top M(B)x\) (from \((a,b,c)\) and \((c,b,a)\)) and eight times to \(\mathcal{C}_{4,n}\) (two cyclic orientations, four starting points). Hence the total contribution of all ordered \(4\)-cycles through \(n\) is \(4x^\top M(B)x\), proving \(\mathcal{C}_{4,n}(B,x)=\mathcal{C}_{4,n-1}(B)+4x^\top M(B)x\).
Multiplying by \(1/8\) and combining with the diagonal quartic split gives the claimed identity.
\end{proof}

With the cumulant identities and the peeling structure in hand, we bound \(Q\) on \(\corereg\). The first estimate controls the \(L^2\) norm of \(e^{4tQ}\); the second shows that \(e^{4tQ}\) is close to~\(1\).

\begin{proposition}[Quartic bounds]\label{prop:quartic-bounds}
There exist absolute constants \(c_1,C_1>0\) such that if
\(d/t\le c_1\) and \(nd^2/t^2\le c_1\), then
\[
  \int_{\corereg}|e^{4tQ(\lambda)}|^2e^{-2t\|\lambda\|^2}\,d\lambda
  \le C_1\core(d,t),
\]
and
\[
  \int_{\corereg}|e^{4tQ(\lambda)}-1|^2e^{-2t\|\lambda\|^2}\,d\lambda
  \le C_1\frac{d^2}{t^2}\core(d,t).
\]
\end{proposition}

\begin{proof}
Let \(A\) be the symmetric zero-diagonal \(n\times n\) matrix with entries \(A_{ij}=\lambda_{\{i,j\}}\), and split off the last vertex:
\[
  A=
  \begin{pmatrix}
    B & x\\
    x^\top & 0
  \end{pmatrix}.
\]
Define
\[
  s(B)=\sum_{1\le a<b\le n-1}B_{ab}^2,
  \qquad
  M(B)=B^2-\diag(B^2).
\]
Since \(\|B\|_F^2=2s(B)\), submultiplicativity gives \(\|M(B)\|_{\mathrm{op}}\le 2s(B)\) and \(\operatorname{tr}(M(B)^2)\le 4s(B)^2\). On \(\corereg\), \(s(B)\le d/t\), so \(\|M(B)\|_{\mathrm{op}}\le 2d/t\) and \(\operatorname{tr}(M(B)^2)\le 4d^2/t^2\).
Since \(Q\) is real, \(|e^{4tQ}|^2=e^{8tQ}\). Fix \(\beta>0\) and set
\[
  I_n^{(\beta)}=\int_{\corereg}e^{-2t\|\lambda\|^2}e^{\beta tQ(\lambda)}\,d\lambda.
\]
By Lemma~\ref{lem:peeling}, \(Q_n(A)\) decomposes into \(Q_{n-1}(B)\) plus terms quadratic and quartic in~\(x\), so for fixed \(B\) with \(s(B)\le d/t\) the \(x\)-fiber contribution is bounded by
\begin{align*}
  &e^{-2ts(B)}e^{\beta tQ_{n-1}(B)}
  \int_{\R^{n-1}}
    \exp\!\left(
      -2t\|x\|^2+\frac{\beta t}{2}x^\top M(B)x-\frac{\beta t}{12}\sum_a x_a^4
    \right)\,dx\\
  &\le
  e^{-2ts(B)}e^{\beta tQ_{n-1}(B)}
  \int_{\R^{n-1}}
    \exp\!\left(
      -2t\|x\|^2+\frac{\beta t}{2}x^\top M(B)x
    \right)\,dx.
\end{align*}
If \(d/t\) is small enough that \(\|(\beta/4)M(B)\|_{\mathrm{op}}\le 1/2\), then \(I-(\beta/4)M(B)\) is positive definite, and Lemma~\ref{lem:gaussian-quadratic} gives
\[
  \int_{\R^{n-1}}
  \exp\!\left(
    -2t\|x\|^2+\frac{\beta t}{2}x^\top M(B)x
  \right)\,dx
  =
  \left(\frac{\pi}{2t}\right)^{(n-1)/2}
  \det\!\left(I-\frac{\beta}{4}M(B)\right)^{-1/2}.
\]
Let \(u_1,\dots,u_{n-1}\) be the eigenvalues of \((\beta/4)M(B)\). Since \(M(B)\) has zero diagonal, \(\sum_j u_j=(\beta/4)\operatorname{tr}M(B)=0\). Also \(|u_j|\le 1/2\), so \(-\log(1-u_j)\le u_j+Cu_j^2\). Therefore
\begin{align*}
  \det\!\left(I-\frac{\beta}{4}M(B)\right)^{-1/2}
  &=\exp\!\left(-\frac12\sum_j\log(1-u_j)\right)\\
  &\le \exp\!\left(C\sum_j u_j^2\right)
   = \exp\!\left(\frac{C\beta^2}{16}\operatorname{tr}(M(B)^2)\right)
  \le \exp\!\left(C_\beta\frac{d^2}{t^2}\right).
\end{align*}
Peeling off vertices one at a time, each step integrates out \(n{-}k\) variables and contributes a Gaussian factor \((\pi/(2t))^{(n-k)/2}\) and a determinant error \(\exp(C_\beta d^2/t^2)\). Since \(\sum_{k=1}^{n-1}(n{-}k)=d\), the accumulated Gaussian factors give \((\pi/(2t))^{d/2}\), and the \(n\) determinant errors multiply to yield \(I_n^{(\beta)}\le \exp(C_\beta nd^2/t^2)(\pi/(2t))^{d/2}\).
If \(nd^2/t^2\le c_1\) with \(c_1\) small enough, then \(I_n^{(\beta)}\le C_\beta F(d,t)\).
By Lemma~\ref{lem:inner-core}, \(F(d,t)\le C\,\core(d,t)\), so taking \(\beta=8\) proves the first estimate.

For the second estimate, \(|e^u-1|^2\le u^2(1+e^{2u})\) with \(u=4tQ(\lambda)\) gives
\begin{align*}
  \int_{\corereg}|e^{4tQ}-1|^2e^{-2t\|\lambda\|^2}\,d\lambda
  &\le 16t^2\int_{\corereg}Q(\lambda)^2e^{-2t\|\lambda\|^2}\,d\lambda\\
  &\quad+
  16t^2\int_{\corereg}Q(\lambda)^2e^{8tQ(\lambda)}e^{-2t\|\lambda\|^2}\,d\lambda.
\end{align*}

For the first term, pass to the centered Gaussian measure \(\gamma_t\) on \(\R^d\) with independent coordinates of variance \((4t)^{-1}\), so that \(\int_{\R^d}f(\lambda)e^{-2t\|\lambda\|^2}\,d\lambda=F(d,t)\E_{\gamma_t}[f(g)]\).
Write \(Q(g)=-\tfrac1{12}D(g)+\tfrac18 C(g)\) with \(D(g)=\sum_e g_e^4\) and \(C(g)=\mathcal{C}_4(g)\).
Then \(\E_{\gamma_t}[D(g)^2]=O(d^2/t^4)\). For the cross term, every monomial in \(D(g)C(g)=\sum_e g_e^4\sum_\omega \prod_{j=1}^4g_{e_j(\omega)}\) has at least one coordinate with odd exponent, so \(\E_{\gamma_t}[D(g)C(g)]=0\).
For \(C(g)^2\), expand into ordered pairs of \(4\)-cycles. A Gaussian moment survives only when every edge variable appears with even total degree, so a fixed ordered cycle has only \(O(1)\) partners that contribute. Since there are \(O(n^4)=O(d^2)\) ordered cycles and each surviving degree-\(8\) Gaussian moment is \(O(t^{-4})\), we get \(\E_{\gamma_t}[C(g)^2]=O(d^2/t^4)\).
Therefore \(\E_{\gamma_t}[Q(g)^2]=O(d^2/t^4)\), and hence
\[
  \int_{\corereg}Q(\lambda)^2e^{-2t\|\lambda\|^2}\,d\lambda
  \le C\frac{d^2}{t^4}\core(d,t).
\]

For the second term, Cauchy--Schwarz gives
\[
  \int_{\corereg}Q(\lambda)^2e^{8tQ(\lambda)}e^{-2t\|\lambda\|^2}\,d\lambda
  \le
  \left(\int_{\corereg}Q(\lambda)^4e^{-2t\|\lambda\|^2}\,d\lambda\right)^{1/2}
  \left(\int_{\corereg}e^{16tQ(\lambda)}e^{-2t\|\lambda\|^2}\,d\lambda\right)^{1/2}.
\]
The second factor is \(O(\core(d,t)^{1/2})\) by the already proved exponential estimate with \(\beta=16\). For the first factor, Lemma~\ref{lem:hc} applied to the degree-\(4\) Gaussian polynomial \(Q(g)\) with \(q=p=4\) yields
\[
  \E_{\gamma_t}[Q(g)^4]
  \le C\,\E_{\gamma_t}[Q(g)^2]^2
  \le C\frac{d^4}{t^8},
  \qquad\text{so}\qquad
  \int_{\corereg}Q(\lambda)^4e^{-2t\|\lambda\|^2}\,d\lambda
  \le C\frac{d^4}{t^8}\core(d,t).
\]
Combining gives \(\int_{\corereg}Q^2e^{8tQ}e^{-2t\|\lambda\|^2}\,d\lambda\le Cd^2t^{-4}\core(d,t)\).
Multiplying by \(16t^2\) proves the second estimate.
\end{proof}

The last cumulant-level ingredient bounds the fifth-order phase \(P\).

\begin{lemma}[Quintic bound]\label{lem:quintic}
There is an absolute constant \(C_3'>0\) such that
\[
  \int_{\corereg}P(\lambda)^2e^{-2t\|\lambda\|^2}\,d\lambda
  \le C_3'\frac{n^5}{t^5}\core(d,t).
\]
\end{lemma}

\begin{proof}
By definition, \(\kappa_5(\lambda)=\E[X_\lambda^5]-10\E[X_\lambda^3]\E[X_\lambda^2]\). The disconnected even networks in \(\E[X_\lambda^5]\) are exactly products of a triangle and a double edge, removed by the subtraction. Thus \(P\) is supported on connected even \(5\)-edge networks.

A connected even multigraph with \(5\) edges has sum of degrees~\(10\) and minimum degree~\(2\), so at most~\(5\) vertices. Choosing at most~\(5\) vertices from~\([n]\) gives \(O(n^5)\) vertex sets; for each, the number of connected even multigraphs with~\(5\) edges on those vertices is bounded by an absolute constant (distribute~\(5\) among at most~\(\binom{5}{2}=10\) edge slots). Therefore \(P\) is a universal linear combination of \(O(n^5)\) monomials.

Under the Gaussian measure \(\gamma_t\) of variance \((4t)^{-1}\), each monomial in \(P\) has degree~\(5\). Fix one such monomial \(M(g)=\prod_e g_e^{a_e}\) with \(\sum_e a_e=5\), and let \(\mathcal O(M)=\{e:a_e\text{ is odd}\}\).
Because the underlying multigraph is even, \(\mathcal O(M)\) is an even simple graph with an odd number of edges; since the total degree is~\(5\), the only possibilities are a triangle or a \(5\)-cycle.

We first describe the two types of monomials~\(M\).
If \(\mathcal O(M)\) is a \emph{\(5\)-cycle}, all five edges carry odd exponent, forcing \(a_e=1\) for each (since \(\sum a_e=5\)); the monomial is fully determined by the cycle.
If \(\mathcal O(M)\) is a \emph{triangle} on~\(\{i,j,k\}\), three edges carry odd exponent (using at least~\(3\) of the~\(5\) slots), and the remaining two slots form a doubled edge.
This doubled edge either coincides with one of the three triangle edges, giving multiplicity pattern~\((3,1,1)\) on the triangle (\(3\)~choices), or joins a triangle vertex to a new vertex~\(\ell\), giving pattern~\((1,1,1,2)\) (\(O(n)\)~choices for~\(\ell\)); see Figure~\ref{fig:triangle-partners}.

Now we count partners. Let \(M'(g)=\prod_e g_e^{b_e}\) be another quintic monomial. If \(\E_{\gamma_t}[M(g)M'(g)]\ne 0\), then \(a_e+b_e\equiv 0\pmod 2\) for every~\(e\), so \(\mathcal O(M')=\mathcal O(M)\).
For a \(5\)-cycle, \(M\) is unique, so the only partner is \(M'=M\).
For a triangle on~\(\{i,j,k\}\), any partner~\(M'\) must share the same odd support and is therefore also of type~\((3,1,1)\) or~\((1,1,1,2)\) on that triangle: \(3 + O(n) = O(n)\) partners.
Since there are \(O(n^5)\) cycle-type monomials and \(O(n^4)\) triangle-type monomials, the total number of nonzero pairs is \(O(n^5)\).

\begin{figure}[ht]
\centering
\begin{tikzpicture}[
  vtx/.style={circle, fill=black, inner sep=1.8pt},
  edge/.style={line width=0.7pt},
  oddedge/.style={line width=0.7pt, blue!70},
  doubledge/.style={line width=0.7pt, orange!80!black},
]

\begin{scope}[shift={(-5,0)}]
  \node[font=\small\bfseries] at (0.65, 2) {Odd support \(\mathcal O(M)\)};
  \node[vtx, label={[font=\scriptsize]above:$i$}] (i) at (0, 1.2) {};
  \node[vtx, label={[font=\scriptsize]below left:$j$}] (j) at (-0.7, 0) {};
  \node[vtx, label={[font=\scriptsize]below right:$k$}] (k) at (1.3, 0) {};
  \draw[oddedge] (i) -- (j) node[midway, left, font=\tiny, blue!70]{odd};
  \draw[oddedge] (j) -- (k) node[midway, below, font=\tiny, blue!70]{odd};
  \draw[oddedge] (i) -- (k) node[midway, right, font=\tiny, blue!70]{odd};
\end{scope}

\begin{scope}[shift={(0,0)}]
  \node[font=\small\bfseries] at (0.65, 2) {Partner type \((3,1,1)\)};
  \node[vtx, label={[font=\scriptsize]above:$i$}] (i2) at (0, 1.2) {};
  \node[vtx, label={[font=\scriptsize]below left:$j$}] (j2) at (-0.7, 0) {};
  \node[vtx, label={[font=\scriptsize]below right:$k$}] (k2) at (1.3, 0) {};
  \draw[oddedge] (i2) -- (j2);
  \draw[oddedge] (i2) -- (k2);
  \draw[oddedge] (j2) -- (k2);
  \draw[doubledge] (j2) to[bend right=18] (k2);
  \draw[doubledge] (j2) to[bend left=18] (k2);
  \node[font=\tiny, orange!80!black] at (0.3, -0.55) {$+2$};
\end{scope}

\begin{scope}[shift={(5,0)}]
  \node[font=\small\bfseries] at (0.65, 2) {Partner type \((1,1,1,2)\)};
  \node[vtx, label={[font=\scriptsize]above:$i$}] (i3) at (0, 1.2) {};
  \node[vtx, label={[font=\scriptsize]below left:$j$}] (j3) at (-0.7, 0) {};
  \node[vtx, label={[font=\scriptsize]below right:$k$}] (k3) at (1.3, 0) {};
  \node[vtx, label={[font=\scriptsize]right:$\ell$}] (l3) at (2.3, 0.6) {};
  \draw[oddedge] (i3) -- (j3);
  \draw[oddedge] (j3) -- (k3);
  \draw[oddedge] (i3) -- (k3);
  \draw[doubledge] (k3) to[bend left=12] (l3);
  \draw[doubledge] (k3) to[bend right=12] (l3);
  \node[font=\tiny, orange!80!black] at (2.2, -0.05) {$+2$};
  \node[font=\tiny, orange!80!black] at (2.75, 0.9) {\(O(n)\) choices};
\end{scope}

\end{tikzpicture}
\caption{Partner counting when \(\mathcal O(M)\) is a triangle on~\(\{i,j,k\}\).
The three blue edges have odd exponent (fixed by~\(\mathcal O\)); the two orange edge-slots carry the remaining multiplicity.
Left: the shared odd support.
Middle: the doubled edge lands on a triangle edge (\(3\) choices).
Right: the doubled edge goes to a new vertex~\(\ell\) (\(O(n)\) choices).}
\label{fig:triangle-partners}
\end{figure}
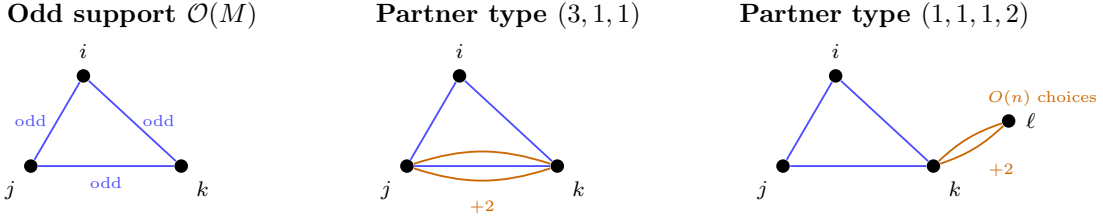

Each surviving product \(M(g)M'(g)=\prod_e g_e^{m_e}\) has \(m_e\in 2\mathbb N_0\) and \(\sum_e m_e=10\). Writing \(g_e=(4t)^{-1/2}Z_e\) with \(Z_e\sim N(0,1)\), we get \(\E_{\gamma_t}[M(g)M'(g)]=(4t)^{-5}\E[\prod_e Z_e^{m_e}]=O(t^{-5})\), since the remaining standard Gaussian moment is bounded by an absolute constant. Multiplying by \(F(d,t)\) and using Lemma~\ref{lem:inner-core} proves the claim.
\end{proof}

\subsection{Expansion of the log-characteristic function}

The estimates in Section~\ref{sec:core} require a factorization of \(\psi(\lambda)^{4t}\) into Gaussian, cubic, quartic, quintic, and remainder terms. Lemma~\ref{lem:sixth} provides this by expanding \(\log\psi(\lambda)\) to fifth order with a sixth-order remainder on \(\{s(\lambda)\le c_2\}\). Lemma~\ref{lem:small-ball} extracts the Gaussian decay \(|\psi(\lambda)|^{4t}\le e^{-3ts(\lambda)/2}\) on a small ball, which controls both the off-core and residual regions.

\begin{lemma}[Log-expansion]\label{lem:sixth}
There are absolute constants \(c_2,C_2>0\) such that whenever \(s(\lambda)\le c_2\),
\[
  \log\psi(\lambda)
  =-\frac12 s(\lambda)-iT(\lambda)+Q(\lambda)+iP(\lambda)+E_6(\lambda)
\]
with \(|E_6(\lambda)|\le C_2 s(\lambda)^3\).
\end{lemma}

\begin{proof}
Set \(m(u)=\E[e^{iuX_\lambda}]\) and \(K(u)=\log m(u)\) for \(0\le u\le 1\).
By Lemma~\ref{lem:hc} applied to the degree-\(2\) Rademacher polynomial \(X_\lambda\)
with \((q,p)=(2,6)\), one has \(\E[|X_\lambda|^6]\le C\,\E[X_\lambda^2]^3=C\,s(\lambda)^3\).
Choose \(c_2\le 1/2\). If \(s(\lambda)\le c_2\), then Lemma~\ref{lem:realpart} applied to \(u\lambda\) gives
\(\RePart m(u)=\RePart\psi(u\lambda)\in[3/4,1]\) for \(0\le u\le 1\).
Hence the principal branch of \(\log\) is well-defined along \(m([0,1])\), and \(|m(u)|\ge 3/4\) on \([0,1]\).

Taylor's theorem with integral remainder gives
\[
  K(1)=\sum_{r=0}^5\frac{K^{(r)}(0)}{r!}+\frac1{5!}\int_0^1(1-u)^5K^{(6)}(u)\,du.
\]
Since \(m(0)=1\) and \(m'(0)=i\E[X_\lambda]=0\),
\(K(0)=K'(0)=0\). Also \(K^{(r)}(0)=i^r\kappa_r(\lambda)\) for \(r\ge 2\). For
\(1\le j\le 6\), one has \(m^{(j)}(u)=i^j\E[X_\lambda^j e^{iuX_\lambda}]\) and
\(|m^{(j)}(u)|\le \E[|X_\lambda|^j]\le \E[|X_\lambda|^6]^{j/6}\) by Jensen.
The sixth derivative of \(K=\log m\) is a finite linear combination of terms
\[
  m(u)^{-k}\prod_{\ell=1}^k m^{(j_\ell)}(u),
  \qquad
  j_1+\cdots+j_k=6.
\]
Since \(|m(u)|\ge 3/4\), each factor \(m(u)^{-k}\) is \(O(1)\), so \(|K^{(6)}(u)|\le C\,\E[|X_\lambda|^6]\le C\,s(\lambda)^3\).
Therefore
\[
  K(1)=\sum_{r=2}^5\frac{i^r}{r!}\kappa_r(\lambda)+O(s(\lambda)^3).
\]
Now \(\kappa_2(\lambda)=s(\lambda)\), \(\kappa_3(\lambda)=6T(\lambda)\),
\(\kappa_4(\lambda)/24=Q(\lambda)\), and \(\kappa_5(\lambda)/120=P(\lambda)\). Substituting these identities proves the claim.
\end{proof}

On the small ball, the quartic and higher corrections are absorbed by the Gaussian decay.

\begin{lemma}[Small-ball decay]\label{lem:small-ball}
There exists an absolute constant \(r_0\in(0,\pi/4)\) such that
\(|\psi(\lambda)|^{4t}\le e^{-3t\|\lambda\|^2/2}\) for every integer \(t\ge 1\)
and every \(\lambda\in\Dd_{r_0}\).
\end{lemma}

\begin{proof}
By Lemma~\ref{lem:sixth}, one has
\(\RePart\log\psi(\lambda)=-\frac12 s(\lambda)+Q(\lambda)+\RePart E_6(\lambda)\)
whenever \(s(\lambda)\le c_2\). It therefore suffices to show
\(Q(\lambda)=O(s(\lambda)^2)\). The diagonal quartic term satisfies
\(\sum_e \lambda_e^4\le s(\lambda)^2\).
For the ordered \(4\)-cycle term, let \(A\) be the symmetric zero-diagonal matrix with entries \(A_{ij}=|\lambda_{\{i,j\}}|\). Then
\[
  |\mathcal{C}_4(\lambda)|
  \le \sum_{i_1,i_2,i_3,i_4}
    A_{i_1i_2}A_{i_2i_3}A_{i_3i_4}A_{i_4i_1}
  = \operatorname{tr}(A^4)
  \le \operatorname{tr}(A^2)^2
  = 4s(\lambda)^2.
\]
Hence \(|Q(\lambda)|\le C s(\lambda)^2\).
Combined with \(|E_6(\lambda)|\le C s(\lambda)^3\) from Lemma~\ref{lem:sixth}, this gives \(\RePart\log\psi(\lambda)\le -\tfrac12 s(\lambda)+C s(\lambda)^2 + C s(\lambda)^3\) on \(\{s(\lambda)\le c_2\}\).
Choose \(r_0\in(0,\pi/4)\) so small that \(r_0^2\le c_2\), \(r_0\le 1\), and \(C r_0^2 + C r_0^4 \le 1/8\). If \(\lambda\in\Dd_{r_0}\), then \(s(\lambda)\le r_0^2\), so
\[
  C s(\lambda)^2 + C s(\lambda)^3
  \le \bigl(Cr_0^2 + Cr_0^4\bigr)s(\lambda)
  \le \tfrac18 s(\lambda),
\]
hence \(\RePart\log\psi(\lambda)\le -3s(\lambda)/8\) and \(|\psi(\lambda)|^{4t}\le e^{-3t s(\lambda)/2}\).
\end{proof}

\subsection{Lower bound on the full core}

With the factorization of \(\psi(\lambda)^{4t}\) from Lemma~\ref{lem:sixth} in hand, we evaluate the core integral by peeling off the higher-order factors one at a time. The cubic phase \(e^{-4itT}\) is the only factor that contributes a non-negligible correction: its oscillation reduces the real part by \(\binom{n}{3}/(8t)\) (Lemma~\ref{lem:cubic}). The quartic, quintic, and sixth-order terms are all absorbed into the error via the bounds above. Proposition~\ref{prop:primary-box} combines these to give the core estimate.

\begin{lemma}[Cubic phase bound]\label{lem:cubic}
For \(t\ge n^3\),
\[
  \RePart K_{\mathrm{core}}(t)
  =
  \left(1-\frac{\binom{n}{3}}{8t}+O\!\left(\frac{n^6}{t^2}+e^{-cd}\right)\right)F(d,t),
\]
where \(K_{\mathrm{core}}(t):=\int_{\corereg}e^{-2t\|\lambda\|^2}e^{-4itT(\lambda)}\,d\lambda\) and \(c>0\) is an absolute constant.
\end{lemma}

\begin{proof}
Because \(\corereg\) is centrally symmetric and \(T(-\lambda)=-T(\lambda)\), the imaginary part of \(K_{\mathrm{core}}(t)\) vanishes. Using \(\cos u=1-u^2/2+O(u^4)\):
\[
  \RePart K_{\mathrm{core}}=\core(d,t)-8t^2 I_2(t)+O(t^4 I_4(t)),
\]
where \(I_k(t):=\int_{\corereg}T(\lambda)^k e^{-2t\|\lambda\|^2}\,d\lambda\).
By Lemma~\ref{lem:triangle}, \(T(\lambda)=\sum_\tau\prod_{e\in\tau}\lambda_e\), summed over the \(\binom n3\) triangles of \(K_n\). Since the coordinates are independent under \(\gamma_t\) and each triangle monomial involves three distinct edges, cross terms with distinct triangle supports vanish:
\[
  \E_{\gamma_t}[T^2]
  =\binom{n}{3}\,\E_{\gamma_t}[\lambda_{e_1}^2]\,\E_{\gamma_t}[\lambda_{e_2}^2]\,\E_{\gamma_t}[\lambda_{e_3}^2]
  =\binom{n}{3}(4t)^{-3}.
\]
Therefore \(8t^2\,\E_{\gamma_t}[T^2]\,F(d,t)=\frac{\binom{n}{3}}{8t}\,F(d,t)\).
Since only diagonal terms survive, and both \(\corereg\) and \(\R^d\setminus\corereg\) are invariant under coordinate sign flips, the off-diagonal cross terms vanish on each region separately, so the restriction from \(\R^d\) to \(\corereg\) can be done term by term. For each triangle \(\tau=\{e_1,e_2,e_3\}\), Cauchy--Schwarz and Lemma~\ref{lem:inner-core} give
\[
  \int_{\R^d\setminus\corereg}\!\lambda_{e_1}^2\lambda_{e_2}^2\lambda_{e_3}^2\,e^{-2t\|\lambda\|^2}\,d\lambda
  \le
  \bigl(\E_{\gamma_t}[\lambda_{e_1}^4\lambda_{e_2}^4\lambda_{e_3}^4]\bigr)^{1/2}\!
  \cdot O(e^{-c_*d/2})\,F(d,t)
  =O(e^{-cd})\,(4t)^{-3}F(d,t),
\]
where the Gaussian fourth moments are computed coordinate by coordinate. Summing over \(\binom{n}{3}\) triangles:
\(8t^2 I_2=\frac{\binom{n}{3}}{8t}\,F(d,t)+O(e^{-cd})\,F(d,t)\).
For \(I_4\), hypercontractivity (Lemma~\ref{lem:hc}, \(q=3,p=4\)) gives
\(\E_{\gamma_t}[T^4]\le C(\E_{\gamma_t}[T^2])^2=C\binom{n}{3}^2(4t)^{-6}\).
Since \(T^4\ge 0\), the core integral is bounded by the full Gaussian expectation: \(I_4\le \E_{\gamma_t}[T^4]\,F(d,t)\), so \(t^4 I_4\le Cn^6 t^{-2}\,F(d,t)\).
Substituting into the expansion for \(\RePart K_{\mathrm{core}}\) and using \(\core(d,t)=F(d,t)(1+O(e^{-c_*d}))\) (Lemma~\ref{lem:inner-core}) yields the claim.
\end{proof}

The core main term now follows by successively removing \(Q\), \(P\), and \(E_6\) from \(\psi(\lambda)^{4t}\) using the bounds above.

\begin{proposition}[Primary estimate]\label{prop:primary-box}
Set \(K_n:=2^{2d-n+1}(2\pi)^{-d}\).
For all sufficiently large \(n\) and \(t\ge C_0n^3\),
\[
  K_n\RePart\int_{\corereg}\psi(\lambda)^{4t}\,d\lambda
  =\left(1-\frac{\binom{n}{3}}{8t}
  +O\!\left(\frac{n^2}{t}+\frac{n^{5/2}}{t^{3/2}}+\frac{n^6}{t^2}+e^{-cd}\right)\right)\Ahat.
\]
\end{proposition}

\begin{proof}
We peel off the factors of \(\psi(\lambda)^{4t}\) one at a time, working from the outermost (sixth-order remainder) inward to the cubic phase. Define the intermediate integrals
\[
  \begin{aligned}
    M_{\mathrm{core}}(t)
    &:=
    \int_{\corereg}
      e^{-2t\|\lambda\|^2}
      e^{-4itT(\lambda)}
      e^{4tQ(\lambda)}
      e^{4itP(\lambda)}
      \,d\lambda,\\
    L_{\mathrm{core}}(t)
    &:=
    \int_{\corereg}
      e^{-2t\|\lambda\|^2}
      e^{-4itT(\lambda)}
      e^{4tQ(\lambda)}
      \,d\lambda.
  \end{aligned}
\]
Since \(d\asymp n^2\) and \(t\ge C_0 n^3\), we have \(d/t\le c_1\) and \(d^3/t^2\le c_1\) for a constant \(c_1=c_1(C_0)\) that can be made arbitrarily small by choosing \(C_0\) large. Also
\(nd^2/t^2=2d^3/((n-1)t^2)\le 2d^3/t^2\),
so for all sufficiently large \(n\) the hypotheses of Lemma~\ref{lem:sixth} and Proposition~\ref{prop:quartic-bounds} hold on \(\corereg\). Hence
\[
  \psi(\lambda)^{4t}
  =e^{-2t\|\lambda\|^2}
    e^{-4itT(\lambda)}
    e^{4tQ(\lambda)}
    e^{4itP(\lambda)}
    e^{4tE_6(\lambda)}
\]
throughout \(\corereg\), with
\(|4tE_6(\lambda)|\le 4C_2t s(\lambda)^3\le Cd^3/t^2\).

We first replace \(\psi(\lambda)^{4t}\) by \(M_{\mathrm{core}}(t)\), absorbing the sixth-order remainder.
If \(n\) is large enough that \(d^3/t^2\) is small, then \(|e^z-1|\le 2|z|\) for \(|z|\le 1\) gives
\(|e^{4tE_6(\lambda)}-1|\le Ct s(\lambda)^3\).
Therefore
\begin{align*}
  &\left|
    \RePart\!\left(\int_{\corereg}\psi(\lambda)^{4t}\,d\lambda\right)-\RePart M_{\mathrm{core}}(t)
  \right|
\le
  \left|
    \int_{\corereg}\psi(\lambda)^{4t}\,d\lambda-M_{\mathrm{core}}(t)
  \right|\\
  &\quad=\left|\int_{\corereg}
    e^{-2t\|\lambda\|^2}
    e^{-4itT(\lambda)}
    e^{4tQ(\lambda)}
    e^{4itP(\lambda)}
    \bigl(e^{4tE_6(\lambda)}-1\bigr)\,d\lambda\right|\\
  &\quad\le
  \left(
    \int_{\corereg}|e^{4tQ(\lambda)}|^2e^{-2t\|\lambda\|^2}\,d\lambda
  \right)^{1/2}
  \left(
    \int_{\corereg}|e^{4tE_6(\lambda)}-1|^2e^{-2t\|\lambda\|^2}\,d\lambda
  \right)^{1/2}.
\end{align*}
By Proposition~\ref{prop:quartic-bounds}, the first factor is \(O(\core(d,t)^{1/2})\). For the second, \(|e^{4tE_6}-1|^2\le Ct^2s(\lambda)^6\), and Lemma~\ref{lem:gaussian-radial} gives \(\int_{\corereg}s(\lambda)^6e^{-2t\|\lambda\|^2}\,d\lambda\le Cd^6t^{-6}\core(d,t)\). Combining:
\[
  \left|
    \RePart\!\left(\int_{\corereg}\psi(\lambda)^{4t}\,d\lambda\right)-\RePart M_{\mathrm{core}}(t)
  \right|
  \le C\frac{d^3}{t^2}\core(d,t).
\]

Next we strip the quintic phase to pass from \(M_{\mathrm{core}}(t)\) to \(L_{\mathrm{core}}(t)\).
Since \(|e^{iu}-1|\le |u|\) for real \(u\), one has
\(|e^{4itP(\lambda)}-1|\le 4t|P(\lambda)|\).
Cauchy--Schwarz, Proposition~\ref{prop:quartic-bounds}, and Lemma~\ref{lem:quintic} yield
\begin{align*}
  |M_{\mathrm{core}}(t)-L_{\mathrm{core}}(t)|
  &\le
  \left(
    \int_{\corereg}|e^{4tQ(\lambda)}|^2e^{-2t\|\lambda\|^2}\,d\lambda
  \right)^{1/2}\\
  &\quad\times
  \left(
    \int_{\corereg}|e^{4itP(\lambda)}-1|^2e^{-2t\|\lambda\|^2}\,d\lambda
  \right)^{1/2}\\
  &\le
  C\core(d,t)^{1/2}\cdot
  t\left(
    \int_{\corereg}P(\lambda)^2e^{-2t\|\lambda\|^2}\,d\lambda
  \right)^{1/2}\\
  &\le C\frac{n^{5/2}}{t^{3/2}}\core(d,t).
\end{align*}

Finally we remove the quartic factor to reach \(K_{\mathrm{core}}(t)\).
By Cauchy--Schwarz and Proposition~\ref{prop:quartic-bounds},
\begin{align*}
  |L_{\mathrm{core}}(t)-K_{\mathrm{core}}(t)|
  &\le
  \left(
    \int_{\corereg}|e^{4tQ(\lambda)}-1|^2e^{-2t\|\lambda\|^2}\,d\lambda
  \right)^{1/2}
  \left(
    \int_{\corereg}e^{-2t\|\lambda\|^2}\,d\lambda
  \right)^{1/2}\\
  &\le C\frac{d}{t}\core(d,t).
\end{align*}

It remains to combine the three bounds. By Lemma~\ref{lem:cubic},
\(\RePart K_{\mathrm{core}}(t)=(1-\binom{n}{3}/(8t)+O(n^6/t^2+e^{-cd}))F(d,t)\).
Combining the three peeling bounds with \(\core(d,t)=(1+O(e^{-c_*d}))F(d,t)\) (Lemma~\ref{lem:inner-core}):
\[
  \int_{\corereg}\RePart(\psi(\lambda)^{4t})\,d\lambda
  =\left(1-\frac{\binom{n}{3}}{8t}+O\!\left(
    \frac{n^2}{t}
    +\frac{n^{5/2}}{t^{3/2}}
    +\frac{n^6}{t^2}
    +e^{-cd}
  \right)\right)F(d,t).
\]
Since \(K_nF(d,t)=\Ahat\), the proposition follows.
\end{proof}

\section{Off-core estimates}\label{sec:residual}

This section shows that the last two terms in~\eqref{eq:three-way} (the off-core integral over \(\Bbox_\delta\setminus\corereg\) and the residual integral over~\(R_\delta\)) are both \(o(\Ahat)\) as \(n\to\infty\) with \(t/(n^{8/3}\log t)\to\infty\). The main tool is a Gaussian comparison inequality (Corollary~\ref{cor:weak-comparison}), which reduces the far-shell bound to a Gaussian characteristic-function estimate. Combined with the core estimate, this yields the expansion in Theorem~\ref{thm:main-intro} (Section~\ref{sec:counting}).

Fix once and for all \(r\in(0,r_0]\), where \(r_0\) is the constant from Lemma~\ref{lem:small-ball}.

\subsection{Weak quadratic comparison}

Lemma~\ref{lem:weak-comparison} compares \(\psi(\lambda)\) with its Gaussian counterpart \(\psi_G(\lambda)\), with an error controlled by the row-sum influences \(I_k(\lambda)\). The proof is the Lindeberg replacement argument from~\cite[Theorem~3.18]{MOO2010}, adapted to retain the per-coordinate bound \(\sum_m\mathrm{Inf}_m(f)^{3/2}\) rather than passing to the maximum influence; this avoids the normalization \(\mathrm{Var}[Q]\le 1\) assumed there. Corollary~\ref{cor:weak-comparison} specializes to the characteristic function setting.

\begin{lemma}[Lindeberg comparison]\label{lem:weak-comparison}
Let \(f:[n]^2\to\R\) be symmetric with \(f(i,i)=0\), and define the quadratic form and row influences
\[
  Q_f(z):=\sum_{1\le i,j\le n}f(i,j)z_i z_j,
  \qquad
  \mathrm{Inf}_m(f):=\sum_{j=1}^n f(m,j)^2.
\]
If \(X_1,\dots,X_n\) are independent Rademacher signs, \(G_1,\dots,G_n\) are i.i.d.\ \(N(0,1)\), and \(\varphi:\R\to\R\) is \(C^3\), then
\[
  \left|\E[\varphi(Q_f(X))]-\E[\varphi(Q_f(G))]\right|
  \le 6\|\varphi'''\|_\infty \sum_{m=1}^n \mathrm{Inf}_m(f)^{3/2}.
\]
\end{lemma}

\begin{proof}
For \(m=0,1,\dots,n\), let
\(Z^{(m)}:=(G_1,\dots,G_m,X_{m+1},\dots,X_n)\) and
\(F_m:=\E[\varphi(Q_f(Z^{(m)}))]\). Then
\(\E[\varphi(Q_f(X))]-\E[\varphi(Q_f(G))]=\sum_{m=1}^n(F_{m-1}-F_m)\).
Fix \(m\). Since \(f\) is symmetric and \(f(m,m)=0\), we may write
\(Q_f(z)=U_m(z)+z_mV_m(z)\), where \(U_m\) and \(V_m\) depend only on coordinates
other than \(m\), and \(V_m(z)=2\sum_{j\ne m}f(m,j)z_j\).
Because \(Z^{(m-1)}\) and \(Z^{(m)}\) agree on every coordinate except the \(m\)-th, evaluating \(U_m\) and \(V_m\) at either hybrid vector gives the same random variables, independent of both \(X_m\) and \(G_m\). Hence
\(F_{m-1}=\E[\varphi(U_m+X_mV_m)]\) and \(F_m=\E[\varphi(U_m+G_mV_m)]\).
Applying Taylor's theorem with integral remainder and using that \(X_m\) and
\(G_m\) are centered with variance \(1\) gives
\[
  |F_{m-1}-F_m|
  \le \frac{\|\varphi'''\|_\infty}{6}\bigl(\E[|X_m|^3]+\E[|G_m|^3]\bigr)\E[|V_m|^3].
\]
Since \(\E[|X_m|^3]=1\) and \(\E[|G_m|^3]=2\sqrt{2/\pi}\), it remains to bound
\(\E[|V_m|^3]\).

Let \(Y_j:=G_j\) for \(j<m\) and \(Y_j:=X_j\) for \(j>m\). Then
\(V_m=2\sum_{j\ne m} f(m,j)Y_j\),
where the \(Y_j\) are independent and centered with \(\E[Y_j^2]=1\).
Hence \(\E[V_m^2]=4\,\mathrm{Inf}_m(f)\).
Since each \(Y_j\) is either Rademacher or standard Gaussian, \(\E[Y_j^4]\le 3\), so
\[
  \E\!\left[\left(\sum_j a_jY_j\right)^4\right]
  =
  \sum_j a_j^4\E[Y_j^4]
  +6\sum_{i<j}a_i^2a_j^2\E[Y_i^2]\E[Y_j^2]
  \le 3\left(\sum_j a_j^2\right)^2,
\]
Setting \(a_j=2f(m,j)\) gives \(\E[V_m^4]\le 48\,\mathrm{Inf}_m(f)^2\). By
Cauchy--Schwarz,
\(\E[|V_m|^3]\le \E[V_m^2]^{1/2}\E[V_m^4]^{1/2}\le 8\sqrt3\,\mathrm{Inf}_m(f)^{3/2}\).
Therefore
\[
  |F_{m-1}-F_m|
  \le
  \frac{1+2\sqrt{2/\pi}}{6}\,8\sqrt3\,\|\varphi'''\|_\infty\,\mathrm{Inf}_m(f)^{3/2}
  \le 6\|\varphi'''\|_\infty\,\mathrm{Inf}_m(f)^{3/2}.
\]
Summing over \(m\) proves the claim.
\end{proof}

Specializing to \(\varphi=\cos\) and \(\varphi=\sin\) yields the comparison for \(\psi\).

\begin{corollary}[Gaussian comparison]\label{cor:weak-comparison}
Let \(\psi_G\) be the characteristic function obtained by replacing the Rademacher entries \(\xi_i\) with i.i.d.\ standard Gaussians \(g_i\). For each row \(k\in[n]\), define the row influence \(I_k(\lambda):=\sum_{i\ne k}\lambda_{\{i,k\}}^2\), and set \(J(\lambda):=\sum_{k=1}^n I_k(\lambda)^{3/2}\). Then
\[
  |\psi(\lambda)-\psi_G(\lambda)|\le \tfrac32 J(\lambda)
  \qquad\text{for every }\lambda\in\R^d.
\]
\end{corollary}

\begin{proof}
Set \(f(i,j)=\lambda_{\{i,j\}}/2\) for \(i\ne j\) and \(f(i,i)=0\), so that
\(Q_f(z)=\sum_{i<j}\lambda_{\{i,j\}}z_i z_j\) and
\(\mathrm{Inf}_m(f)=\tfrac14 I_m(\lambda)\).
Applying Lemma~\ref{lem:weak-comparison} with \(\varphi=\cos\) and
\(\varphi=\sin\), and using \(\|\cos'''\|_\infty,\|\sin'''\|_\infty\le 1\), gives
\[
  |\psi(\lambda)-\psi_G(\lambda)|
  \le 12\sum_{m=1}^n \mathrm{Inf}_m(f)^{3/2}
  =\frac32 J(\lambda).
\]
\end{proof}

\subsection{Pointwise far-shell bounds}

The far shell \(\Bbox_{\pi/4}\setminus\Dd_r\) is split into three sub-regions according to the comparison error \(J(\lambda)\) and the maximal row-sum \(I_{\max}(\lambda)\). On each, we obtain a strict contraction of \(|\psi|\).

\begin{proposition}[Far-shell contraction]\label{prop:far-shell}
Set \(\eta_r:=\tfrac{1}{3}[1-(1+2r^2)^{-1/4}]\) and \(I_{\max}(\lambda):=\max_k I_k(\lambda)\). There exist constants \(q_{\mathrm{sm}},q_{\mathrm{big}}\in(0,1)\) and \(a_r>0\), depending only on~\(r\), such that for all \(n\ge 2\) and \(t\ge 1\) the far shell decomposes as
\[
  \begin{aligned}
  \mathcal F_{\mathrm{sm}}
    &:=\{\lambda\in\Bbox_{\pi/4}\setminus\Dd_r:J(\lambda)\le \eta_r\},\\
  \mathcal F_{\mathrm{big}}
    &:=\{\lambda\in\Bbox_{\pi/4}\setminus\Dd_r:J(\lambda)>\eta_r,\;I_{\max}(\lambda)\ge 1\},\\
  \mathcal F_{\mathrm{mid}}
    &:=\{\lambda\in\Bbox_{\pi/4}\setminus\Dd_r:J(\lambda)>\eta_r,\;I_{\max}(\lambda)<1\},
  \end{aligned}
\]
and
\[
  |\psi(\lambda)|^{4t}\le
  \begin{cases}
    q_{\mathrm{sm}}^{4t} & \lambda\in\mathcal F_{\mathrm{sm}},\\[3pt]
    q_{\mathrm{big}}^{4t} & \lambda\in\mathcal F_{\mathrm{big}},\\[3pt]
    \exp(-a_r t n^{-2/3}) & \lambda\in\mathcal F_{\mathrm{mid}}.
  \end{cases}
\]
\end{proposition}

\begin{proof}
Set \(q_G:=(1+2r^2)^{-1/4}\), \(q_{\mathrm{sm}}:=(1+q_G)/2\), \(q_{\mathrm{big}}:=((1+e^{-8/\pi^2})/2)^{1/2}\), and \(a_r:=(4/\pi^2)\eta_r^{2/3}\).

\emph{Small-\(J\) region.}
Let \(M\) be the symmetric matrix with \(M_{ij}=\lambda_{\{i,j\}}\) for \(i\ne j\)
and \(M_{ii}=0\), so that \(\sum_{i<j}\lambda_{\{i,j\}}g_ig_j=\tfrac12 g^\top Mg\).
Lemma~\ref{lem:gaussian-quadratic} gives
\[
  \psi_G(\lambda)=\det(I-iM)^{-1/2},
  \qquad
  |\psi_G(\lambda)|=\det(I+M^2)^{-1/4}.
\]
Because \(M^2\) is positive semidefinite, \(\det(I+M^2)\ge 1+\operatorname{tr}(M^2)=1+2s(\lambda)\), so
\[
  |\psi_G(\lambda)|\le (1+2s(\lambda))^{-1/4}\le q_G
  \qquad
  (\lambda\in\Bbox_{\pi/4}\setminus\Dd_r).
\]
If \(\lambda\in\mathcal F_{\mathrm{sm}}\), then Corollary~\ref{cor:weak-comparison} gives
\(|\psi(\lambda)-\psi_G(\lambda)|\le \tfrac32 J(\lambda)\le \tfrac32\eta_r=(1-q_G)/2\).
Therefore \(|\psi(\lambda)|\le q_G+(1-q_G)/2=q_{\mathrm{sm}}\), and hence
\(|\psi(\lambda)|^{4t}\le q_{\mathrm{sm}}^{4t}\) for \(\lambda\in\mathcal F_{\mathrm{sm}}\).

\emph{Large-\(I_{\max}\) region.}
If \(\lambda\in\mathcal F_{\mathrm{big}}\), choose \(k\) with \(I_k(\lambda)=I_{\max}(\lambda)\).
By Lemma~\ref{fact:psi-sq},
\[
  |\psi(\lambda)|^2\le \tfrac12+\tfrac12\prod_{i\ne k}\cos(2\lambda_{\{i,k\}}).
\]
Since \(|\lambda_{\{i,k\}}|\le \pi/4\), the elementary bound
\(\cos(2x)\le 1-8x^2/\pi^2\le e^{-8x^2/\pi^2}\) for \(|x|\le \pi/4\)
implies
\[
  |\psi(\lambda)|^2
  \le
  \frac12+\frac12\exp\!\left(-\frac{8}{\pi^2}I_k(\lambda)\right)
  \le
  \frac12+\frac12 e^{-8/\pi^2}
  =q_{\mathrm{big}}^2.
\]
Thus \(|\psi(\lambda)|\le q_{\mathrm{big}}\), and hence
\(|\psi(\lambda)|^{4t}\le q_{\mathrm{big}}^{4t}\) for \(\lambda\in\mathcal F_{\mathrm{big}}\).

\emph{Middle region.}
If \(\lambda\in\mathcal F_{\mathrm{mid}}\), the same bound gives
\[
  |\psi(\lambda)|^2
  \le
  \tfrac12+\tfrac12\exp\!\bigl(-\tfrac{8}{\pi^2}I_{\max}(\lambda)\bigr).
\]
Set \(x=8I_{\max}(\lambda)/\pi^2\in[0,8/\pi^2)\subset[0,1)\). The function \(f(x):=e^{x/4}+e^{-3x/4}\) satisfies \(f'(x)<0\) on \([0,1]\), so \(f(x)\le f(0)=2\). Equivalently, \((1+e^{-x})/2\le e^{-x/4}\), giving \(|\psi(\lambda)|^2\le e^{-2I_{\max}(\lambda)/\pi^2}\).
Since \(J(\lambda)\le nI_{\max}(\lambda)^{3/2}\), the condition
\(\lambda\in\mathcal F_{\mathrm{mid}}\) gives
\(I_{\max}(\lambda)\ge (\eta_r/n)^{2/3}\).
Hence
\[
  |\psi(\lambda)|^{4t}
  \le
  \exp\!\left(-\frac{4}{\pi^2}tI_{\max}(\lambda)\right)
  \le
  \exp\!\left(-a_r t n^{-2/3}\right)
  \qquad (\lambda\in\mathcal F_{\mathrm{mid}}).
\]
\end{proof}

Integrating the three pointwise bounds over the far shell and comparing with \(\Ahat\) yields the following.

\begin{corollary}[Far-shell integrals]\label{cor:far-negligible}
Let \(q_{\mathrm{sm}},q_{\mathrm{big}},a_r\) be the constants from
Proposition~\ref{prop:far-shell}, and put
\(E_r(n,t):=q_{\mathrm{sm}}^{4t}+q_{\mathrm{big}}^{4t}+e^{-a_r t n^{-2/3}}\).
Then for every \(\delta\in(0,\pi/4)\),
\[
  \begin{aligned}
  2^{2d-n+1}(2\pi)^{-d}\int_{\Bbox_\delta\setminus\Dd_r}|\psi(\lambda)|^{4t}\,d\lambda
    &\le 2^{-n+1}E_r(n,t),\\[3pt]
  (2\pi)^{-d}\sum_{\lambda\in\Lambda}
    \int_{\lambda+(\Bbox_{\pi/4}\setminus\Dd_r)}|\psi(\gamma)|^{4t}\,d\gamma
    &\le 2^{-n+1}E_r(n,t).
  \end{aligned}
\]
In particular, both quantities are \(o(\Ahat)\) as \(n\to\infty\) with \(t/(n^{8/3}\log t)\to\infty\).
\end{corollary}

\begin{proof}
Since \(\Bbox_\delta\setminus\Dd_r\subset \Bbox_{\pi/4}\setminus\Dd_r\), integrating
the three pointwise bounds from Proposition~\ref{prop:far-shell} over
\(\Bbox_\delta\setminus\Dd_r\) gives
\[
  \int_{\Bbox_\delta\setminus\Dd_r}|\psi(\lambda)|^{4t}\,d\lambda
  \le
  \left(\frac{\pi}{2}\right)^d E_r(n,t).
\]
Multiplying by \(2^{2d-n+1}(2\pi)^{-d}(\pi/2)^d=2^{-n+1}\)
proves the first inequality.
For the translated far shell, let \(S:=\Bbox_{\pi/4}\setminus\Dd_r\). Then for
\(\lambda\in\Lambda\) and \(\mu\in S\subset\Bbox_{\pi/4}\), Lemma~\ref{fact:lambda-facts}(i),(iv) gives
\(|\psi(\lambda+\mu)|^{4t}=|\psi(\mu)|^{4t}\). Therefore
\[
  (2\pi)^{-d}\sum_{\lambda\in\Lambda}\int_{\lambda+S}|\psi(\gamma)|^{4t}\,d\gamma
  =2^{2d-n+1}(2\pi)^{-d}\int_S|\psi(\mu)|^{4t}\,d\mu,
\]
and the same integrated bound over \(S\) yields the second inequality.

It remains to show \(2^{-n+1}E_r(n,t)=o(\Ahat)\) as \(n\to\infty\) with \(t/(n^{8/3}\log t)\to\infty\); all asymptotic notation below refers to this limit. Fix \(q\in(0,1)\). Then
\[
  \log\frac{2^{-n+1}q^{4t}}{\Ahat}
  =-2d\log2+\tfrac d2\log(8\pi t)+4t\log q.
\]
Since \(d\asymp n^2\) and \(t/(n^{8/3}\log t)\to\infty\), one has \(t/(d\log t)\to\infty\), so the
negative linear term dominates and \(2^{-n+1}q^{4t}=o(\Ahat)\). For the exponential term,
\[
  \log\frac{2^{-n+1}e^{-a_r t n^{-2/3}}}{\Ahat}
  =-2d\log2+\tfrac d2\log(8\pi t)-a_r t n^{-2/3}.
\]
Since \(t/(n^{8/3}\log t)\to\infty\), we have \(tn^{-2/3}/(d\log t)\to\infty\), so the negative linear term dominates and \(2^{-n+1}e^{-a_r t n^{-2/3}}=o(\Ahat)\). Summing the three terms proves the final assertion.
\end{proof}

\subsection{Residual estimate}

Proposition~\ref{prop:residual-estimate} combines the small-ball decay (Lemma~\ref{lem:small-ball}), the integrated far-shell bound (Corollary~\ref{cor:far-negligible}), and the odd-cell bound (Lemma~\ref{lem:odd-bound}) to show that the off-core and residual integrals are both negligible.

\begin{proposition}[Residual estimate]\label{prop:residual-estimate}
Set \(\delta^2=2d/t\) and \(K_n:=2^{2d-n+1}(2\pi)^{-d}\). Then
\[
  \begin{aligned}
    K_n\int_{\Bbox_\delta\setminus\corereg}|\psi(\lambda)|^{4t}\,d\lambda &= o(\Ahat),\\[3pt]
    (2\pi)^{-d}\left|\int_{R_\delta}\psi(\gamma)^{4t}\,d\gamma\right| &= o(\Ahat),
  \end{aligned}
\]
as \(n\to\infty\) with \(t/(n^{8/3}\log t)\to\infty\).
Moreover, for \(t\ge C_0 n^3\) and all sufficiently large~\(n\), both bounds are at most \(Ce^{-cn^2}\Ahat\) for absolute constants \(c,C>0\).
\end{proposition}

\begin{proof}
All asymptotic notation below is as \(n\to\infty\) with \(t/(n^{8/3}\log t)\to\infty\).
Since \(d\asymp n^2\), we have \(d/t\to 0\), so eventually \(\delta<r\). Also
\(K_n\core(d,t)=[1+o(1)]\Ahat\)
by the proof of Proposition~\ref{prop:primary-box}.

\emph{Outside the core.}
Set \(\mathcal A_{\mathrm{loc}}(t,\delta,r):=(\Bbox_\delta\cap\Dd_r)\setminus\corereg\).
If \(\lambda\in\mathcal A_{\mathrm{loc}}(t,\delta,r)\), then \(\|\lambda\|^2>d/t\), so Lemma~\ref{lem:small-ball} gives
\[
  |\psi(\lambda)|^{4t}
  \le e^{-3t\|\lambda\|^2/2}
  \le e^{-d/2}e^{-t\|\lambda\|^2}.
\]
Therefore
\[
  \int_{\mathcal A_{\mathrm{loc}}(t,\delta,r)}|\psi(\lambda)|^{4t}\,d\lambda
  \le e^{-d/2}\left(\frac{\pi}{t}\right)^{d/2}
  \le Ce^{-c_4 d}\core(d,t),
  \qquad
  c_4=\frac12-\frac{\log2}{2}>0,
\]
and hence \(K_n\int_{\mathcal A_{\mathrm{loc}}(t,\delta,r)}|\psi(\lambda)|^{4t}\,d\lambda=o(\Ahat)\).
By Corollary~\ref{cor:far-negligible},
\(K_n\int_{\Bbox_\delta\setminus\Dd_r}|\psi(\lambda)|^{4t}\,d\lambda=o(\Ahat)\). Since
\(\Bbox_\delta\setminus\corereg
=\mathcal A_{\mathrm{loc}}(t,\delta,r)\sqcup(\Bbox_\delta\setminus\Dd_r)\),
this proves the first claim.

\emph{Odd residual.}
By Lemma~\ref{lem:residual-decomp}, \(R_\delta=R_\delta^{\mathrm{even}}\cup R_\delta^{\mathrm{odd}}\); we bound each piece separately.
For \(\gamma\in R_\delta^{\mathrm{odd}}\), Lemma~\ref{lem:odd-bound} gives
\(|\psi(\gamma)|^2\le \frac12\), so \(|\psi(\gamma)|^{4t}\le 2^{-2t}\). Hence
\[
  (2\pi)^{-d}\int_{R_\delta^{\mathrm{odd}}}|\psi(\gamma)|^{4t}\,d\gamma\le 2^{-2t}.
\]
As in the proof of Corollary~\ref{cor:far-negligible}, \(2^{-2t}=o(\Ahat)\).

\emph{Translated near shell.}
Since
\(\Bbox_{\pi/4}\setminus\Bbox_\delta
\subset
\bigl(\Dd_r\setminus\Bbox_\delta\bigr)\cup\bigl(\Bbox_{\pi/4}\setminus\Dd_r\bigr)\),
the even residual is bounded by a translated near-shell term plus a translated far-shell term. For the translated near shell, if \(\mu\in\Dd_r\setminus\Bbox_\delta\), then some coordinate of \(\mu\) has absolute value at least \(\delta\), so \(\|\mu\|^2\ge \delta^2\), and Lemma~\ref{lem:small-ball} gives
\(|\psi(\mu)|^{4t}\le e^{-t\delta^2/2}e^{-t\|\mu\|^2}\).
Hence
\[
  \int_{\Dd_r\setminus\Bbox_\delta}|\psi(\mu)|^{4t}\,d\mu
  \le e^{-t\delta^2/2}\left(\frac{\pi}{t}\right)^{d/2}.
\]
For \(\lambda\in\Lambda\) and \(\mu\in\Dd_r\setminus\Bbox_\delta\subset\Bbox_{\pi/4}\), Lemma~\ref{fact:lambda-facts}(i),(iv) gives
\(|\psi(\lambda+\mu)|^{4t}=|\psi(\mu)|^{4t}\).
Therefore, since \(t\delta^2/2=d\),
\[
  (2\pi)^{-d}\sum_{\lambda\in\Lambda}\int_{\lambda+(\Dd_r\setminus\Bbox_\delta)}|\psi(\gamma)|^{4t}\,d\gamma
  \le
  2^{2d-n+1}(2\pi)^{-d}e^{-d}\left(\frac{\pi}{t}\right)^{d/2}
  = e^{-c_* d}\Ahat,
\]
where \(c_*=1-\frac12\log 2\) is the constant from Lemma~\ref{lem:inner-core}.

\emph{Translated far shell.}
Finally, the second inequality of Corollary~\ref{cor:far-negligible} gives \(o(\Ahat)\) for the translated far-shell contribution.
Summing the odd, translated near-shell, and translated far-shell bounds proves the second claim.

\emph{Exponential bound for \(t\ge C_0 n^3\).}
We verify each bound is \(O(e^{-cn^2})\Ahat\).
The local annulus and translated near-shell are already expressed as \(Ce^{-c_4 d}\core(d,t)\) and \(e^{-c_*d}\Ahat\), respectively; since \(d=\binom{n}{2}\ge n^2/3\), both are \(O(e^{-cn^2})\Ahat\).

For the remaining pieces, use the lower bound
\(\Ahat\ge c_0\,2^{d-n+1}(2\pi t)^{-d/2}\),
which follows from \(K_nF(d,t)=(1+O(e^{-c_*d}))\Ahat\).

Corollary~\ref{cor:far-negligible} bounds both far-shell contributions by \(2^{-n+1}E_r(n,t)\).
Since \(q_{\mathrm{sm}},q_{\mathrm{big}}<1\) are absolute constants, for~\(n\) large enough that \(4\log(1/q_{\mathrm{sm}}),4\log(1/q_{\mathrm{big}})\ge a_r n^{-2/3}\) one has \(E_r(n,t)\le 3\exp(-a_r t n^{-2/3})\), giving
\[
  \frac{2^{-n+1}E_r}{\Ahat}
  \le C\,(\pi t)^{d/2}\exp\!\bigl(-a_r t n^{-2/3}\bigr).
\]
The exponent \(a_r t n^{-2/3}-(d/2)\log(\pi t)\) is increasing for \(t\ge dn^{2/3}/(2a_r)\asymp n^{8/3}\), and at \(t=C_0 n^3\) it equals \(a_rC_0 n^{7/3}-O(n^2\log n)\ge c\,n^{7/3}\), so the ratio is \(O(e^{-cn^2})\).

The odd residual satisfies \(2^{-2t}/\Ahat\le C\,(2\pi t)^{d/2}\,2^{-2t-d+n-1}\). Since \((2t+d-n)\log 2\) grows linearly in \(t\) while \((d/2)\log(2\pi t)\) grows logarithmically, the ratio is \(O(e^{-cn^3})\le O(e^{-cn^2})\) for \(t\ge C_0n^3\).
\end{proof}

\section{Proof of main theorem}\label{sec:counting}

We now combine the primary estimate (Proposition~\ref{prop:primary-box}) and the off-core bounds (Proposition~\ref{prop:residual-estimate}) to prove Theorem~\ref{thm:main-intro}.

\begin{proof}[Proof of Theorem~\ref{thm:main-intro}]
Fix \(n\) sufficiently large and \(t\ge C_0 n^3\).
By the three-way decomposition~\eqref{eq:three-way},
\[
  \Prb(S_{4t}=0)
  =K_n\RePart\!\int_{\corereg}\psi^{4t}\,d\lambda
  +\RePart E,
\]
where \(E\) collects the off-core and residual contributions.
Proposition~\ref{prop:primary-box} gives
\[
  K_n\RePart\!\int_{\corereg}\psi(\lambda)^{4t}\,d\lambda
  =\left(1-\frac{\binom{n}{3}}{8t}
  +O\!\left(\frac{n^2}{t}+\frac{n^{5/2}}{t^{3/2}}+\frac{n^6}{t^2}+e^{-cn^2}\right)\right)\Ahat.
\]
Proposition~\ref{prop:residual-estimate} gives \(|E|=O(e^{-cn^2})\Ahat\), which is absorbed by the existing \(O(e^{-cn^2})\) term.
Since \(N_{n,4t}=2^{4nt}\Prb(S_{4t}=0)\) and \(A_{n,4t}=2^{4nt}\Ahat\), dividing by~\(\Ahat\) gives the theorem.
\end{proof}

\begin{corollary}[Uniform counting]\label{cor:uniform}
For every \(\varepsilon>0\) there exists \(K>0\) such that
\[
  \left|\frac{N_{n,4t}}{A_{n,4t}}-1\right|<\varepsilon
  \qquad\text{for all } n\ge 2 \text{ and } t\ge Kn^3.
\]
\end{corollary}

\begin{proof}
For \(n\ge N_0\) (with \(N_0\) large enough that \(Ce^{-cN_0^2}<\varepsilon/2\)) and \(t\ge Kn^3\), the expansion gives \(|N_{n,4t}/A_{n,4t}-1|\le \binom{n}{3}/(8t)+Cn^2/t+Cn^{5/2}/t^{3/2}+Cn^6/t^2+Ce^{-cn^2}\le C'/K+\varepsilon/2<\varepsilon\) for \(K\) large. For each \(n<N_0\), Lemma~\ref{fact:fixed-n-dl10} provides \(T_n\) such that \(|N_{n,4t}/A_{n,4t}-1|<\varepsilon\) for \(t\ge T_n\); enlarging \(K\) so that \(Kn^3\ge T_n\) for all \(n<N_0\) completes the proof.
\end{proof}

Thus \(N_{n,4t}\sim A_{n,4t}\) as \(t/n^3\to\infty\). Conversely, when \(t=\Theta n^3\) with \(\Theta\ge C_0\) fixed, the error is \(O(\Theta^{-2}+1/n+e^{-cn^2})\), so for \(\Theta\) large and \(n\to\infty\) the leading correction \(1/(48\Theta)\) dominates and \(N_{n,4t}/A_{n,4t}\) does not converge to~\(1\): the asymptotics of \(N_{n,4t}\) change at the cubic scale.

\bibliographystyle{abbrv}
\bibliography{refs}

\end{document}